\documentclass[12pt]{amsart}

\usepackage[dvipsnames]{xcolor}

\usepackage[normalem]{ulem}

\usepackage{hyperref} 
\hypersetup{
    colorlinks,
    linkcolor=OliveGreen,
    citecolor=OliveGreen,
    urlcolor=OliveGreen
	}

\usepackage{amsmath}
\usepackage{amssymb}
\usepackage{mathrsfs}
\usepackage{amscd}
\usepackage{graphicx}
\usepackage{mathdots}

\usepackage{gensymb}
\usepackage{epstopdf}
 \usepackage{todonotes}
 \usepackage{youngtab}
 \usepackage{wrapfig}
 \usepackage{float}

 \usepackage{bbm}
 \usepackage{bm}
 \usepackage{verbatim}
 \usepackage[margin=0.9in]{geometry}
 \usepackage[nocompress]{cite}
 \usepackage{epsfig}       
\usepackage{epic,eepic}        

\usepackage{tikz}
\usetikzlibrary{arrows.meta}

\usepackage{mathpazo}
\usepackage{domitian}
\usepackage[T1]{fontenc}

\usepackage{thmtools, thm-restate}

\newtheorem{thm}{Theorem}[section]
\newtheorem{prop}[thm]{Proposition}
\newtheorem{proposition}[thm]{Proposition}
\newtheorem{lem}[thm]{Lemma}
\newtheorem{cor}[thm]{Corollary}

\theoremstyle{definition}
\newtheorem{definition}[thm]{Definition}
\newtheorem{example}[thm]{Example}

\theoremstyle{remark}
\newtheorem{remark}[thm]{Remark}

\numberwithin{equation}{section}

\newcommand{\tr}{\operatorname{tr}}

\newcommand{\E}{\mathbb{E}}

\newcommand{\btQ}{\mathbf{\tilde{Q}}}

\newcommand{\R}{\mathbb{R}}

\newcommand{\tQ}{\tilde{Q}}

\begin{document}

\title{On the limiting Horn inequalities}

\author{Samuel G. G. Johnston}
\address{Department of Mathematics, King's College London}
\email{samuel.g.johnston@kcl.ac.uk}

 \author{Colin McSwiggen}
\address{Institute of Mathematics, Academia Sinica}
\email{csm@gate.sinica.edu.tw}

\maketitle

\begin{abstract}

The Horn inequalities characterise the possible spectra of triples of $n$-by-$n$ Hermitian matrices $A+B=C$.  We study integral inequalities that arise as limits of Horn inequalities as $n \to \infty$.  These inequalities are parametrised by the points of an infinite-dimensional convex body, the asymptotic Horn system $\mathscr{H}[0,1]$, which can be regarded as a topological closure of the countable set of Horn inequalities for all finite $n$.

We prove three main results.  The first shows that arbitrary points of $\mathscr{H}[0,1]$ can be well approximated by specific sets of finite-dimensional Horn inequalities. Our second main result shows that $\mathscr{H}[0,1]$ has a remarkable self-characterisation property. That is, membership in $\mathscr{H}[0,1]$ is determined by the very inequalities corresponding to the points of $\mathscr{H}[0,1]$ itself. 
To illuminate this phenomenon, we sketch a general theory of sets that characterise themselves in the sense that they parametrise their own membership criteria, and we consider the question of what further information would be needed in order for this self-characterisation property to determine the Horn inequalities uniquely.
Our third main result is a quantitative result on the redundancy of the Horn inequalities in an infinite-dimensional setting.  Concretely, the Horn inequalities for finite $n$ are indexed by certain sets $T^n_r$ with $1 \le r \le n-1$; we show that if $(n_k)_{k \ge 1}$ and $(r_k)_{k \ge 1}$ are any sequences such that $(r_k / n_k)_{k \ge 1}$ is a dense subset of $(0,1)$, then the Horn inequalities indexed by the sets $T^{n_k}_{r_k}$ are sufficient to imply all of the others.

\end{abstract}

\section{Introduction}

\subsection{Background}

Horn's problem is a fundamental question in linear algebra: what are the possible spectra of three $n$-by-$n$ Hermitian matrices satisfying $A + B = C$?  The answer to this question was first conjectured by Horn \cite{Ho} and proved in a culmination of works by multiple authors, most notably Klyachko \cite{Kly} and Knutson and Tao \cite{KT99}.  Briefly, given $\alpha, \beta, \gamma \in \R^n$, there is a recursive procedure for enumerating a system of linear inequalities on the triple $(\alpha,\beta,\gamma) \in \R^{3n}$, and these inequalities are a necessary and sufficient condition for the existence of three such matrices $A+B=C$ with $\alpha, \beta, \gamma$ as their respective spectra. More explicitly, for each $1 \leq r \leq n-1$ there is a collection $T_r^n$ of triples of subsets $(I,J,K)$ of $\{1,\ldots,n\}$ of the same cardinality $|I| = |J|= |K| = r$ with the following property:

\begin{thm}[\cite{Kly, KT99}] \label{thm:horn}
    For any $\alpha, \beta, \gamma \in \R^n$ with their coordinates listed in nonincreasing order, the following are equivalent:
    \begin{enumerate}
        \item There exist $n$-by-$n$ Hermitian matrices $A$ and $B$ such that $(A,B,A+B)$ have respective\\
        spectra 
        $(\alpha,\beta,\gamma)$.
        \item $\sum_{i=1}^n \alpha_i + \sum_{j=1}^n \beta_j = \sum_{k=1}^n \gamma_k$, and the inequalities 
\begin{equation} \label{eqn:hornform}
    \sum_{i \in I} \alpha_i + \sum_{j \in J} \beta_j \ge \sum_{k \in K} \gamma_k
\end{equation}
 hold for all triples $(I,J,K) \in T^n_r$ for all $1 \le r \le n-1$.
    \end{enumerate}
\end{thm}

We refer to the inequalities \eqref{eqn:hornform} as the \textbf{Horn inequalities}. 

The collections $T_r^n$ indexing the Horn inequalities are constructed recursively. 
One starts by defining $T_1^n = \{ (\{i\},\{j\},\{k\}) \ | \ 1 \leq i,j,k \leq n, \ i+j =k+1 \}$. Next, given a subset $I = \{i_1 < \ldots < i_r\}$ of $\{1,\ldots,n\}$, let $\lambda(I) \in \mathbb{R}^r$ be the vector $\lambda(I) = (i_r-r,\ldots,i_1-1)$. 
Then for $r \geq 2$,
the triples of subsets $(I,J,K)$ belonging to $T_r^n$ --- that is, the triples that correspond to Horn inequalities --- are precisely those for which $(\lambda(I), \lambda(J), \lambda(K))$ \textit{solves} the Horn inequalities indexed by $T_p^r$ for $1 \leq p \leq r-1$.

Simultaneously with the solution of the finite-dimensional Horn's problem, various authors began to consider infinite-dimensional versions that deal with operators on a Hilbert space.  Here there are a number of ways to pose the problem by placing different stipulations on the operators and their spectral data.  One strand of work has dealt with the case of compact operators \cite{Friedland, Fu2, BLS05, BLT09}, while many other papers have considered the setting of von Neumann algebras, specifically finite factors \cite{BL01, BL06, BCDLT, ColDyk08, ColDyk09, ColDyk11}.  In this latter case, one associates to each self-adjoint operator a probability measure on its spectrum (its $*$-distribution), and one asks what $*$-distributions can arise for operators $A+B=C$. The solution is described by a system of infinitely many inequalities on an infinite-dimensional space of triples of probability measures.  Specifically, one can write the original Horn inequalities for $n$-by-$n$ matrices as integral inequalities on the empirical spectral measures; then the possible $*$-distributions for self-adjoint operators $A + B = C$ are those that simultaneously satisfy the Horn inequalities for all finite $n$ \cite{BL06,BCDLT}.

In this paper, we take a different perspective on the asymptotics of Horn's problem as $n \to \infty$.  Here we sidestep questions about spectra of operators and focus instead on the inequalities themselves.  Our investigation begins from three striking observations about the infinite system of Horn inequalities mentioned above.

The first observation is that there are \textit{too many} inequalities.  The list enumerated by Horn's algorithm is known to be enormously redundant, in that there are proper subsets of the Horn inequalities that imply all of the others.   This is true even in finite dimensions when $n \ge 5$, but in infinite dimensions the phenomenon is more dramatic due to the fact that any given Horn inequality is implied asymptotically by approximation with sequences of other inequalities.  In particular, the set of solutions does not change if any finite number of inequalities are discarded.  We would like to know which subsets of inequalities are sufficient to characterise the problem and which can be safely ignored without changing the solutions.  In the infinite-dimensional setting, such subsets can have a completely different character than in finite dimensions.

The second observation is that there are \textit{too few} inequalities.  Solutions to the Horn inequalities also solve many other inequalities that are not in the list.  For example, suppose that $(\pi_1,\pi_2,\pi_3)$ are a triple of compactly supported probability measures satisfying the Horn inequalities, and write $Q_i: [0,1] \to \R$ for the quantile function of $\pi_i$. (We recall the definition of quantile functions in Section \ref{sec:defs}.) 
Then the Horn inequalities imply the limiting Ky Fan inequalities:
\begin{equation} \label{eqn:lim-KF}
\int_x^1 Q_1(t) + Q_2(t) - Q_3(t) \, \mathrm{d} t \ge 0 \qquad \forall \, x \in [0,1].
\end{equation}
However, the Horn inequalities themselves are countable and only include the cases where $x \in  \mathbb{Q} \cap (0,1)$.  Clearly, therefore, it is both possible and desirable to take some kind of closure or completion of the Horn inequalities in order to obtain a much larger list of criteria that the solutions must satisfy.  But in precisely what sense can we take such a closure while ensuring that the resulting inequalities remain valid for all solutions to the original system?  This turns out to be a more delicate issue than the simple example of \eqref{eqn:lim-KF} might suggest.

The third observation is that the set of all Horn inequalities \textit{describes itself} in a peculiar sense. As explained above, the triples of subsets $(I,J,K)$ corresponding to Horn inequalities are precisely those for which $(\lambda(I), \lambda(J), \lambda(K))$ \textit{solves} the Horn inequalities.
In finite dimensions this fact is explicit in Horn's recursive procedure, but in infinite dimensions this property takes a stranger form: each infinite-dimensional Horn inequality is indexed by a triple of probability measures that solves \textit{all} of the other inequalities, not merely those for some fixed $n$.  Thus the full set of inequalities encodes a set of necessary and sufficient conditions for membership in itself.  It is natural to ask, does this self-characterisation property still hold for the ``completion'' of the Horn inequalities discussed above?  And, how close does this property come to uniquely determining the Horn inequalities?

Here we gain new insights into all of the above questions by identifying the Horn inequalities with points in an infinite-dimensional space and studying their topological closure, which represents the set of integral inequalities on triples of probability measures that can be obtained as scaling limits of Horn inequalities as $n$ grows large.  This set can be regarded either as a system of inequalities or as a geometric object in its own right, and we consider it from both perspectives. 

Concretely, the central object of our investigation is the \textit{asymptotic Horn system} $\mathscr{H}[0,1]$, which consists of all triples of probability measures supported on $[0,1]$ that can be obtained as weak limits of empirical spectral measures of $n$-by-$n$ Hermitian triples $A_n + B_n = C_n$.
As we explain, points of $\mathscr{H}[0,1]$ correspond to families of integral inequalities that arise as limits of Horn inequalities in a suitable topology. Thus $\mathscr{H}[0,1]$ can be regarded as the closure of the Horn inequalities when viewed ``from infinity.''

Our first main result says, intuitively, that arbitrary elements of $\mathscr{H}[0,1]$ can be approximated by certain subsets of $\mathscr{H}[0,1]$ corresponding to finite-dimensional Horn inequalities with a prescribed limiting ratio $r/n$ as $n$ grows large. To obtain this result, we identify each set $T^n_r$ with a collection of points in $\mathscr{H}[0,1]$ and bound the distance in a suitable metric between arbitrary points of $\mathscr{H}[0,1]$ and points corresponding to elements of $T^n_r$ for specific values of $r$ and $n$.

Our second main result shows that $\mathscr{H}[0,1]$ indeed has a self-characterisation property.  That is, similarly to the original countable collection of Horn inequalities, their closure $\mathscr{H}[0,1]$ also describes itself: the points of $\mathscr{H}[0,1]$ are in correspondence with a set of necessary and sufficient conditions that determine whether a given triple of probability measures belongs to $\mathscr{H}[0,1]$.  Put differently, elements of $\mathscr{H}[0,1]$ are solutions of the asymptotic Horn inequalities, but in another sense they also \textit{are} the asymptotic Horn inequalities.  This property is an asymptotic residue of the recursive structure of Horn's problem in finite dimensions, whereby the inequalities for $n$-by-$n$ matrices can be constructed from the inequalities for $r$-by-$r$ matrices with $r \le n-1$.

The self-characterisation property also answers the question of how the Horn inequalities can be ``completed'': it implies that any triple of compactly supported probability measures that satisfies the Horn inequalities in fact satisfies all inequalities corresponding to points in $\mathscr{H}[0,1]$.  It is not obvious that such a property would persist after taking the topological closure: although the \textit{solutions} to the Horn inequalities form a closed set, the inequalities themselves do not, and they are defined in terms of a functional that is not continuous in the relevant topology.  Therefore it is a nontrivial fact that the solutions of the original system actually satisfy all of the inequalities in the closure.

As a consequence of the first two main results described above, we obtain our third main result, which gives a quantitative bound on the redundancy in the Horn inequalities as $n \to \infty$, in the following sense. 
Let $(r_k)_{k \ge 1}$ and $(n_k)_{k \ge 1}$ be any sequences of positive integers such that $(r_k/n_k)_{k \ge 1}$ is a dense subset of $(0,1)$.  Then the inequalities indexed by the sets $T^{n_k}_{r_k}$ for $k \ge 1$ imply the full system of Horn inequalities.  This result is ``quantitative'' because it tells us about the \textit{rate} at which Horn inequalities may be discarded without affecting the solutions of the system --- or, more precisely, it shows that there is no bound on the rate, in that $n_k$ can grow arbitrarily quickly and $r_k/n_k$ can approximate points of $(0,1)$ arbitrarily slowly.

At the end of the article, we ask: what additional information, if any, would be needed in order for the self-characterisation property to \textit{uniquely} determine the Horn inequalities?  To illuminate this question and provide some partial answers, we sketch a general theory of sets that characterise themselves, in the sense that elements of the set parametrise a collection of necessary and sufficient conditions for membership in the set.  We prove some abstract but elementary results on the existence, uniqueness, and structure of such sets in a very general setting, and we then apply these to the particular case of the Horn inequalities.

In the course of our analysis we will draw extensively on the theorems and techniques of authors such as Knutson and Tao \cite{KT99} and Bercovici and Li \cite{BL06}. We stress that we do not aim to reprove their results, but rather to apply them to new asymptotic questions about the Horn inequalities, which we hope will offer new insights into the finite-dimensional setting as well.

\subsection{Overview}

The outline of the paper is as follows.

Section \ref{sec:prelims} presents our main results and introduces definitions and notation that we will use throughout the article.  In particular, we define our main object of study, the asymptotic Horn system $\mathscr{H}[0,1]$, as well as a variant, the extended asymptotic Horn system $\mathscr{H}(\mathbb{R})$, and we discuss some of their basic properties.

In Section \ref{sec:hornineq} we review the Horn inequalities for $n$-by-$n$ matrices, and then reformulate them in an infinite-dimensional setting as integral inequalities on triples of quantile functions of probability measures, rather than linear inequalities on triples of vectors.  This yields an intrinsic statement of the Horn inequalities that applies independently of the value of $n$ and is thus well-suited for studying asymptotic problems.

Section \ref{sec:main-proofs} contains the proofs of our main theorems on self-characterisation and approximation of elements of $\mathscr{H}[0,1]$, as well as some corollaries and auxiliary results.

In Section \ref{sec:SCunique} we begin to develop a general theory of self-characterising sets, and we consider the question of how close the self-characterisation property comes to determining the Horn inequalities uniquely.

Finally, in Section \ref{sec:further}, we discuss some possible directions for further research.

\section{Main results} \label{sec:prelims}

\subsection{Definitions}  \label{sec:defs}
The goal of this article is to prove some fundamental properties of the \textit{asymptotic Horn system} $\mathscr{H}[0,1]$, which is the set of triples of probability measures supported on $[0,1]$ that occur as the limiting empirical spectra of Hermitian matrices $A_1+A_2=A_3$. 

We now define $\mathscr{H}[0,1]$ precisely. We say that  a probability measure $\pi$ is $n$-atomic if it can be written $\pi = \frac{1}{n} \sum_{i =1}^n \delta_{x_i}$ for some $x_1,\ldots,x_n \in \mathbb{R}$.
Given an $n$-by-$n$ Hermitian matrix $A$, its empirical spectral measure is the $n$-atomic probability measure
\[
\pi_A := \frac{1}{n} \sum_{i=1}^n \delta_{\alpha_i},
\]
where $\alpha_1 \geq \hdots \geq \alpha_n \in \R$ are the eigenvalues of $A$. 

\begin{definition} \label{def:main}
We will be interested in the following collections of triples of probability measures:
\begin{enumerate}
\item The collection $\mathscr{H}_n(\mathbb{R})$ of triples $(\pi_1,\pi_2,\pi_3)$ of $n$-atomic probability measures $(\pi_1,\pi_2,\pi_3)$ that are the respective empirical spectra of $n$-by-$n$ Hermitian matrices $A_1,A_2,A_3$ satisfying $A_1 + A_2 = A_3$. \vspace{2.5mm}
\item The \textbf{extended asymptotic Horn system} is defined by
\begin{align*}
\mathscr{H}(\mathbb{R}) := \text{Closure of } \bigcup_{n \geq 1} \mathscr{H}_n(\mathbb{R})
\end{align*}
in the product topology obtained from the Wasserstein metric on probability measures with finite expectation (see \eqref{eqn:W1-def} below).
\vspace{2.5mm}
\item For Borel subsets $E \subseteq \mathbb{R}$, let $\mathcal{P}(E)^3$ denote the set of triples of probability measures supported in $E$. We define
\begin{align*}
\mathscr{H}(E)  &:= \mathscr{H}(\mathbb{R}) \cap \mathcal{P}(E)^3\\
& := \{ (\pi_1,\pi_2,\pi_3) \in \mathscr{H}(\mathbb{R}) \ | \ \text{Each $\pi_i$ is supported in $E$} \}. 
\end{align*}
We will be particularly interested in the case where $E$ is the unit interval $[0,1]$, in which case we write $\mathscr{H}[0,1] := \mathscr{H}([0,1])$ and refer to this object as the \textbf{asymptotic Horn system}. 

\end{enumerate}
\end{definition}

We note that the support of a probability measure is a closed set by definition.
Thus $\mathscr{H}[0,1)$ consists of triples of measures that are supported on $[0,1-\varepsilon]$ for some $\varepsilon >0$.

We recall that
for two probability measures $\pi, \pi'$ on $\R$ with finite expectation, their Wasserstein distance is defined by
\begin{equation} \label{eqn:W1-def}
    W_1(\pi, \pi') := \inf_{\gamma \in \Gamma(\pi,\pi')} \E_{(X,Y) \sim \gamma} \big[|Y-X|\big],
\end{equation}
where $\Gamma(\pi,\pi')$ is the space of couplings of $\pi$ and $\pi'$. Convergence in Wasserstein distance is equivalent to weak convergence plus convergence in expectation. 
A triple of probability measures $(\pi_1,\pi_2,\pi_3)$ lies in $\mathscr{H}(\mathbb{R})$ if and only if there exists a sequence of triples $(\pi_{1,k},\pi_{2,k},\pi_{3,k})_{k \geq 1}$ with $(\pi_{1,k},\pi_{2,k},\pi_{3,k})$ in some $\mathscr{H}_{n_k}(\mathbb{R})$, such that $\lim_{k \to \infty} W_1(\pi_{i,k},\pi_i) = 0$ for $i=1,2,3$.

We refer to $\mathscr{H}[0,1]$ as the asymptotic Horn \textit{system} for two reasons.  The first reason we will discuss in detail below: the points of $\mathscr{H}[0,1]$ are in correspondence with a system of integral inequalities that can be regarded as a metric completion or topological closure of the Horn inequalities. (The same is not true for all points of $\mathscr{H}(\mathbb{R})$.)  The second reason is to distinguish the objects $\mathscr{H}[0,1]$ and $\mathscr{H}(\mathbb{R})$ from the related \textit{asymptotic Horn bodies} that have been studied elsewhere in the literature \cite{ColDyk08, ColDyk11}, which are sets of the form 
\begin{equation} \label{eqn:asymp-body-def}
    \mathscr{H}_{\mu,\nu} := \big\{ (\pi_1, \pi_2, \pi_3) \in \mathscr{H}(\mathbb{R}) \ | \ \pi_1 = \mu, \, \pi_2 = \nu \big\}
\end{equation}
for $\mu, \nu$ two fixed probability measures.
 
Observe that, given Hermitian matrices $A + B = C$, we have not only $(\pi_A, \pi_B, \pi_C) \in \mathscr{H}_n(\mathbb{R})$ but additionally $(\pi_{sA+tI}, \pi_{sB+uI}, \pi_{sC+(t+u)I}) \in \mathscr{H}_n(\mathbb{R})$ as well for all $s,t,u \in \R$, where $I$ is the $n$-by-$n$ identity matrix.  It follows that $\mathscr{H}(\mathbb{R})$ is invariant under transformations of the form $(\pi_1, \pi_2, \pi_3) \mapsto (T_\#^t D_\#^s \pi_1, T_\#^u D_\#^s \pi_2, T_\#^{t+u} D_\#^s \pi_3)$, where $D^s(x) = sx$ and $T^t(x) = x+t$ for $x \in \R$, and $S_\# \pi$ indicates the pushforward of a measure $\pi$ under a measurable map $S:\mathbb{R} \to \mathbb{R}$, i.e. $S_\# \pi(E) := \pi(S^{-1}(E))$. Moreover, any compactly supported triple in $\mathscr{H}(\mathbb{R})$ can be obtained from a triple in $\mathscr{H}[0,1]$ via such a transformation.  Thus it is quite natural and nonrestrictive that we consider only measures supported on $[0,1]$ in the definition of $\mathscr{H}[0,1]$.

Our results deal with various conditions that characterise the triples $(\pi_1, \pi_2, \pi_3)$ belonging to $\mathscr{H}[0,1]$ or $\mathscr{H}(\mathbb{R})$.  One can observe immediately that by the trace equality $\tr(A) + \tr(B) = \tr(C)$, any such triple must satisfy
\begin{align} \label{eq:trace1}
0 = - \int_{-\infty}^\infty x \, \pi_1(\mathrm{d}x) - \int_{-\infty}^\infty x \, \pi_2(\mathrm{d}x)  +  \int_{-\infty}^\infty x \, \pi_3(\mathrm{d}x).
\end{align}
There are, of course, far more demanding requirements a triple of probability measures must meet in order to guarantee membership in $\mathscr{H}[0,1]$ or $\mathscr{H}(\mathbb{R})$. It is most natural to state these requirements in terms of {\textbf{quantile functions}. Given a probability measure $\pi$ on $\mathbb{R}$, its quantile function $Q:[0,1] \to \mathbb{R} \cup \{\pm \infty\}$ is the unique right-continuous nondecreasing function satisfying
\begin{align*}
t = \int_{-\infty}^{Q(t)}\pi(\mathrm{d}x) \qquad \text{for $t \in [0,1]$}.
\end{align*}
If $\pi$ has a strictly increasing and continuous cumulative distribution function $F$, then $Q = F^{-1}$.  Otherwise the quantile function may be discontinuous if $\pi$ has gaps in its support, and it may be constant over open intervals if $\pi$ has atoms. If $\pi$ is supported in an interval $[a,b]$, then $Q(t)$ takes values in $[a,b]$. The quantile function has the property that for measurable $f$ we have
\begin{align} \label{eq:fx}
\int_{-\infty}^\infty f(x) \, \pi(\mathrm{d}x) = \int_0^1 f(Q(t)) \, \mathrm{d}t.
\end{align}
In other words, $\pi$ is the pushforward by $Q$ of Lebesgue measure on $[0,1]$.

A quantile function $Q$ associated with a probability measure $\pi$ is said to be $\textbf{integrable}$ if it is an element of $L^1([0,1])$, i.e.\ if $\int_0^1|Q(t)|\, \mathrm{d}t < \infty$. By \eqref{eq:fx}, this is equivalent to $\int_{-\infty}^\infty |x| \, \pi(\mathrm{d}x) < \infty$. If $Q$ and $Q'$ are integrable quantile functions associated with probability measures $\pi$ and $\pi'$, then with $W_1(\pi,\pi')$ as in \eqref{eqn:W1-def} it is a straightforward calculation (see e.g.\ \cite[Proposition 2.17]{sant}) to show that
\begin{align*}
W_1(\pi,\pi') = ||Q'-Q||_1 := \int_0^1 |Q'(t)-Q(t)| \, \mathrm{d}t.
\end{align*}
Since the correspondence between probability measures and their quantile functions is bijective, we may equally think of $\mathscr{H}[0,1]$ or $\mathscr{H}(\mathbb{R})$ as consisting of triples $\mathbf{Q} = (Q_1,Q_2,Q_3)$ of quantile functions. Everywhere below, the relevant notion of convergence for such triples can be understood as Wasserstein convergence of probability measures in each entry or, equivalently, $L^1$ convergence of quantile functions. We will freely use the identification of measures with their quantile functions in the sequel.

\vspace{2mm}
In the next section, we describe the relationship between each $\mathscr{H}_n(\mathbb{R})$ and $\mathscr{H}(\mathbb{R})$, and we indicate how the Horn inequalities themselves can be regarded as elements of $\mathscr{H}[0,1]$. Then, in the following sections, we state our main results:
\begin{itemize} 
\item Theorem \ref{thm:dense} is our main approximation result. It states that any point of $\mathscr{H}[0,1]$ can be obtained as a limit of Horn inequalities belonging to sets $T^n_r$ such that the ratio $r/n$ tends to a prescribed value.
\item Theorem \ref{thm:SC} establishes the self-characterisation property. It states that membership in $\mathscr{H}(\R)$ is characterised by a system of integral inequalities indexed by the points of $\mathscr{H}[0,1)$, and that membership in $\mathscr{H}[0,1]$ is characterised by a system of integral inequalities indexed by the points of $\mathscr{H}[0,1]$ itself.
\item Our final main result, Theorem \ref{thm:redundancy}, concerns the asymptotic redundancy in the Horn inequalities. It combines Theorem \ref{thm:dense} and Theorem \ref{thm:SC} to show that certain ``small'' subsets of Horn inequalities in fact imply the entire system.
\end{itemize}

\subsection{Embedding the Horn inequalities into $\mathscr{H}[0,1]$}
\label{sec:embed}

In this section we describe how $\mathscr{H}_n(\mathbb{R})$ embeds as a subset of $\mathscr{H}(\mathbb{R})$, and we show that each Horn inequality of the form \eqref{eqn:hornform} may be associated with an element of $\mathscr{H}[0,1]$.  We will be particularly concerned with two special kinds of quantile functions:

\begin{definition} \label{def:atomic}
A quantile function $Q:[0,1] \to \mathbb{R}$ is said to be:
\begin{enumerate}
\item \textbf{$n$-atomic} if it is constant on each interval of the form $
[(i-1)/n,i/n)$ for $1 \leq i \leq n$;
\item \textbf{$n$-integral} if it takes values in the integer multiples of $1/n$.
\end{enumerate}
\end{definition}

Equivalently, an $n$-atomic quantile function is simply the quantile function of an $n$-atomic probability measure, and an $n$-integral quantile function is the quantile function of a probability measure supported on $\frac{1}{n}\mathbb{Z}$.

Note that with $\mathscr{H}_n(\mathbb{R})$ and $\mathscr{H}(\mathbb{R})$ as in Definition \ref{def:main} we have 
\begin{align} \label{eq:incld}
\mathscr{H}_n(\mathbb{R}) \subseteq \{ \text{$n$-atomic elements of $\mathscr{H}(\mathbb{R})$} \}.
\end{align}
In fact we will see below that this inclusion is actually an equality.

Given a subset $I = \{i_1 < \ldots < i_r\} \subseteq \{1,\ldots,n\}$, we define an associated quantile function $Q_{I,n}:[0,1] \to \mathbb{R}$ by setting
\begin{align} \label{eq:Iquant}
Q_{I,n}(t) = \frac{i_s-s}{n}, \qquad t \in \left[ \frac{s-1}{r} , \frac{s}{r} \right), \quad s \in \{1,2,\ldots,r\},
\end{align}
with the convention $Q_{I,n}(1) := \lim_{t \uparrow 1} Q_{I,n}(t)$. Note that $Q_{I,n}$ is $n$-integral and $r$-atomic and takes values in $[0,1-r/n]$.

Recall that Theorem \ref{thm:horn} states that $\mathscr{H}_n(\mathbb{R})$ is characterised by a collection $T_r^n$ of linear inequalities on the eigenvalues of the three matrices $A,B,C$. The following lemma, proved below in Section \ref{sec:selfavg}, embeds both the set $\mathscr{H}_n(\mathbb{R})$ of spectra of $n$-by-$n$ Hermitian triples \textit{and} the set $T^n_r$ of Horn inequalities as subsets of $\mathscr{H}(\mathbb{R})$.

\begin{lem}[Embedding lemma] \label{lem:embed}
We have:
\begin{enumerate}
\item The inclusion \eqref{eq:incld} is in fact an equality, that is, $\mathscr{H}_n(\mathbb{R}) = \{ \text{$n$-atomic elements of $\mathscr{H}(\mathbb{R})$} \}$.
\item The map sending $(I,J,K)$ to $(Q_{I,n},Q_{J,n},Q_{K,n})$ (with $Q_{I,n}$ as in \eqref{eq:Iquant}) is a bijection from $T_r^n$ to the collection of triples in $\mathscr{H}[0,1-r/n]$ that are both $n$-integral and $r$-atomic.
\end{enumerate}
\end{lem}

In light of \eqref{eq:incld}, the first part of Lemma \ref{lem:embed} states that if $(\pi_1,\pi_2,\pi_3)$ is a triple of $n$-atomic measures occuring as a weak limit of a sequence $(\pi_{1,k},\pi_{2,k},\pi_{3,k})_{k \geq 1}$ of triples with $(\pi_{1,k},\pi_{2,k},\pi_{3,k})$ lying in some $\mathscr{H}_{n_k}(\mathbb{R})$, then $(\pi_1,\pi_2,\pi_3)$ lies in $\mathscr{H}_n(\mathbb{R})$. The second part of Lemma \ref{lem:embed} embeds $T_r^n$ into $\mathscr{H}[0,1]$. (Note however that the map $T^n_r \to \mathscr{H}[0,1]$ is only guaranteed to be injective for each \textit{fixed} choice of $n$ and $r$; the images of different sets $T_r^n$ and $T_{q}^{m}$ may not be disjoint.
)

In short, Lemma \ref{lem:embed} allows us to make the following identifications, which put the eigenvalues of Hermitian triples and the Horn inequalities themselves on equal footing:
\begin{align}
\mathscr{H}_n(\mathbb{R})  &=  \{ \text{$n$-atomic elements of $\mathscr{H}(\mathbb{R})$} \} \nonumber \\
&= \{ \text{empirical spectra of $n$-by-$n$ triples} \}, \label{eq:assoc0} \\
 \mathscr{H}_r^n[0,1-r/n] &:= \{ \text{$n$-integral and $r$-atomic elements of $\mathscr{H}[0,1-r/n]$} \} \nonumber \\
&\cong \{ \text{Horn inequalities indexed by $T_r^n$} \}.\label{eq:assoc}
\end{align}
These identifications license us to speak of the Horn inequalities as elements of $\mathscr{H}(\mathbb{R})$, and as such, make sense of notions such as ``approximating an element of $\mathscr{H}(\mathbb{R})$ by Horn inequalities.''

\begin{figure}
\centering
\begin{tikzpicture}[scale=0.5]
\draw[->] (0,0) -- (10.5,0);
\draw[->] (0,0) -- (0,10.5);
\draw[gray] (10,0) -- (10,10) -- (0,10);
\draw[gray] (0,0.83333) -- (10,0.83333);
\draw[gray] (0,1.666666666666) -- (10,1.666666666666);
\draw[gray] (0,2.5) -- (10,2.5);
\draw[gray] (0,3.3333333) -- (10,3.3333333);
\draw[gray] (0,4.166666666) -- (10,4.166666666);
\draw[gray] (0,5) -- (10,5);
\draw[gray] (0,5.8333) -- (10,5.8333);
\draw[gray] (0,6.6666) -- (10,6.6666);
\draw[gray] (0,7.5) -- (10,7.5);
\draw[gray] (0,8.3333) -- (10,8.3333);
\draw[gray] (0,9.16666) -- (10,9.16666);

\draw[gray] (2,0) -- (2,10);
\draw[gray] (4,0) -- (4,10);
\draw[gray] (6,0) -- (6,10);
\draw[gray] (8,0) -- (8,10);

\node[below=2pt of {(2,0)}] {$1/r$};
\node[below=2pt of {(4,0)}] {$2/r$};
\node[below=2pt of {(6,0)}] {$\ldots$};
\node[below=2pt of {(10,0)}] {$r/r$};

\node[left=2pt of {(0,0.83333)}] {$1/n$};
\node[left=2pt of {(0,1.666666666666)}] {$2/n$};
\node[left=2pt of {(0,3.5)}] {$\vdots$};
\node[left=2pt of {(0,10)}] {$n/n$};

\node[BrickRed, below right=2pt of {(5.7,2.96666666666)}] {$Q_{J,n}$};

\draw[BrickRed, very thick] (0,0.81333) -- (2,0.80333); 
\draw[BrickRed, thick] (2,0.80333) -- (2,1.646666);
\draw[BrickRed, very thick] (2,1.646666) -- (4,1.646666);
\draw[BrickRed, thick] (4,1.646666) -- (4,2.5);
\draw[BrickRed, very thick] (4,2.5) -- (6,2.5);
\draw[BrickRed, thick] (6,2.5) -- (6,3.33333);
\draw[BrickRed, very thick] (6,3.33333) -- (8,3.33333);
\draw[BrickRed, thick] (8,3.333333) -- (8,5);
\draw[BrickRed, very thick] (8,5) -- (10,5);

\draw[Tan, very thick] (0,0) -- (4,0);
\draw[Tan, thick] (4,0) -- (4,0.83333);
\draw[Tan, very thick] (4,0.83333) -- (8, 0.83333);
\draw[Tan, thick] (8,0.83333) -- (8,2.5);
\draw[Tan, very thick] (8,2.5) -- (10,2.5);

\node[Tan, below right=0pt of {(8,1.23333)}] {$Q_{I,n}$};

\draw[OliveGreen, very thick] (0,0.86333) -- (1.95,0.86333);
\draw[OliveGreen, thick] (1.95,0.86333) -- (1.95, 1.68666666);
\draw[OliveGreen, very thick] (1.95,1.686666) -- (3.95,1.686666);
\draw[OliveGreen, thick] (3.95,1.666666) -- (3.95,3.33333);
\draw[OliveGreen, very thick] (3.95,3.33333) -- (6,3.33333);
\draw[OliveGreen, thick] (6,3.33333) -- (6,5.83333);
\draw[OliveGreen, very thick] (6,5.833333) -- (10,5.833333);

\node[OliveGreen, above=0pt of {(5,3.333)}] {$Q_{K,n}$};

\end{tikzpicture}
\caption{Three $n$-integral and $r$-atomic quantile functions with $n=12$ and $r = 5$. These are associated with the Horn triple
\begin{align*}
I = \{ 1, 2, 4, 5, 8 \}   , \quad J = \{ 2,4,6,8,11 \} \quad \text{and} \quad K = \{ 2, 4, 7, 11, 12 \}.
\end{align*}
Lemma \ref{lem:embed} states that the $n$-integral and $r$-atomic quantile functions lying in $\mathscr{H}[0,1-r/n]$ are in bijection with the sets $T_r^n$ indexing the Horn inequalities.}

\end{figure}

\subsection{The approximation theorem} \label{sec:approx-thm}
We have just seen that the Horn inequalities may be identified with certain points of $\mathscr{H}[0,1]$. We now can state our first main result on approximating arbitrary points of $\mathscr{H}[0,1]$ by specific sequences of Horn inequalities.

For a triple $\mathbf{Q} = (Q_1,Q_2,Q_3)$ of quantile functions taking values in $[0,1]$, define
\begin{equation}
    \eta_{\mathbf{Q}} := 1 - \max_{i=1,2,3} \sup_{t\in[0,1]} Q_i(t).
\end{equation}
That is, if each $Q_i$ is the quantile function of a measure $\pi_i$, then $\eta_{\mathbf{Q}}$ is equal to the smallest distance between the support of any $\pi_i$ and the right endpoint of the unit interval.

Our approximation result states that any element $\mathbf{Q}$ of $\mathscr{H}[0,1]$ may be approximated by a sequence of Horn inequalities with any given asymptotic ratio $r_n/n \to q \in [0,\eta_\mathbf{Q}]$.

\begin{restatable}{thm}{densethm} \label{thm:dense}
Let $\mathbf{Q} \in \mathscr{H}[0,1]$. Then for every $0 \le q \le \eta_{\mathbf{Q}}$  there exists a sequence $((I_n,J_n,K_n))_{n \geq 1}$ such that $(I_n,J_n,K_n) \in T^n_{r_n}$ with $r_n/n \to q$ and $(Q_{I_n,n},Q_{J_n,n},Q_{K_n,n}) \to \mathbf{Q}$. 
\end{restatable}

In fact we deduce Theorem \ref{thm:dense} from a more detailed result, which is one of the key technical tools in the paper (Theorem \ref{thm:new23}). It gives a quantitative bound on the distance between points of $\mathscr{H}[0,1]$ and triples representing Horn inequalities in specific sets $T^n_r$.

Let us highlight the special case $q = 0$ of Theorem \ref{thm:dense}, which says that any element at all of $\mathscr{H}[0,1]$ may be approximated by a sequence of Horn inequalities $(I_n,J_n,K_n) \in T^n_{r_n}$ with $r_n/n \to 0$. Thus while $\mathscr{H}[0,1]$ was originally defined as the set of triples of $[0,1]$-supported probability measures that occur as weak limits of empirical spectra of $n$-by-$n$ Hermitian triples $A_1 + A_2 = A_3$, Theorem \ref{thm:dense} offers the following alternative perspective on $\mathscr{H}[0,1]$.

\begin{cor} \label{cor:closure}
The asymptotic Horn system $\mathscr{H}[0,1]$ is the topological closure of the set
\[
\bigcup_{n=2}^{\infty} \bigcup_{r=1}^{n-1} \mathscr{H}_r^n[0,1-r/n].
\]
This latter set is the image of all finite-dimensional Horn inequalities under the correspondence \eqref{eq:assoc}.
\end{cor}

\begin{remark}
Readers familiar with the work of Bercovici and Li \cite{BL06} may be aware that they gave an $n$-independent formulation of the Horn inequalities by identifying the collection of inequalities for all finite $n$ with a set $\mathcal{T}$ of triples of indicator functions of subsets of $[0,1]$. Our scheme for embedding the Horn inequalities in an infinite-dimensional space is different from that used in \cite{BL06}. Nevertheless, it can be deduced from Theorem \ref{thm:dense} that if one normalises the indicator functions in \cite{BL06} so that they become densities of probability measures, then the asymptotic Horn system $\mathscr{H}[0,1]$ is the closure of $\mathcal{T}$ in the topology of weak convergence.
\end{remark}

\subsection{The self-characterisation theorem} \label{sec:SC}

As we explain in detail in Section \ref{sec:hornineq}, each point $\mathbf{Q} \in \mathscr{H}[0,1]$ gives rise to a family of integral inequalities on triples of quantile functions.  For any fixed choice of $\mathbf{Q}$ these inequalities are parametrised by a value $\mu \in [0,\eta_{\mathbf{Q}}]$, and when $\mathbf{Q}$ corresponds to a triple $(I,J,K) \in T^n_r$ via \eqref{eq:assoc}, we can recover the corresponding Horn inequality by taking $\mu = r/n$.

Our second main result, the self-characterisation theorem, says that membership in $\mathscr{H}(\mathbb{R})$ or $\mathscr{H}[0,1]$ is equivalent to satisfying all integral inequalities parametrised by elements of $\mathscr{H}[0,1)$ or $\mathscr{H}[0,1]$ respectively.  That is, the theorem has two parts.  The first part says that if a triple of quantile functions $\mathbf{Q}$ satisfies the trace equality \eqref{eq:trace}, then $\mathbf{Q} \in \mathscr{H}(\R)$ if and only if $\mathbf{Q}$ satisfies all integral inequalities corresponding to any $\btQ \in \mathscr{H}[0,1)$ and any $\mu \in (0,\eta_\btQ]$.  The second part says that, moreover, if each $Q_i$ takes values in $[0,1]$, then the same statement holds with $\mathscr{H}[0,1]$ in place of $\mathscr{H}[0,1)$ and with $\mu \in [0,\eta_\btQ]$ rather than $(0,\eta_\btQ]$.

We use the term \textit{self-characterisation} because of this second conclusion, which says that $\mathscr{H}[0,1]$ is determined by a set of inequalities indexed by its own elements.  This property arises in the $n \to \infty$ limit as a consequence of the recursive structure of the Horn inequalities.  As we show in Corollary \ref{cor:cpt-supp}, it also confirms that if a triple of compactly supported measures satisfies the countable family of finite-$n$ Horn inequalities, then this triple actually satisfies all inequalities parametrised by the \textit{closed} set $\mathscr{H}[0,1]$.  This provides an answer to the question of how the Horn inequalities can be ``completed.''

It may not be immediately obvious what the second part of the theorem adds to the first part.  After all, since $\mathbf{Q} \in \mathscr{H}[0,1]$ if and only if $\mathbf{Q} \in \mathscr{H}(\R)$ and each $Q_i$ takes values in $[0,1]$, the inequalities indexed by $\btQ \in \mathscr{H}[0,1)$ are already sufficient to determine membership in $\mathscr{H}[0,1]$ as well.  The point is that if $\mathbf{Q} \in \mathscr{H}[0,1]$, then $\mathbf{Q}$ \textit{additionally} satisfies the inequalities indexed by all $\btQ \in \mathscr{H}[0,1] - \mathscr{H}[0,1)$. This statement is a nontrivial extension of the first part of the theorem.  We now make all of this precise.

Given two triples of quantile functions $\mathbf{Q} = (Q_1,Q_2,Q_3)$ and $\btQ = (\tQ_1,\tQ_2,\tQ_3)$ with each $\tQ_i$ taking values in $[0,1]$, define the composition functional 
\begin{align} \label{eq:energy}
\mathcal{E}(\mathbf{Q},\mathbf{\tQ}) := \int_0^1 - Q_1(\tQ_1(t)) - Q_2(\tQ_2(t)) + Q_3(\tQ_3(t)) \, \mathrm{d}t,
\end{align}
provided the integral in \eqref{eq:energy} is defined; if this integral is undefined, then so is $\mathcal{E}(\mathbf{Q},\mathbf{\tQ})$. The composition functional plays a fundamental role throughout this article. It is linear in its first argument, in that for any triples $\mathbf{P}$ and $\mathbf{Q}$ of quantile functions and any $a,b \in \R$, we have
\begin{align} \label{eq:linear}
\mathcal{E}( a \mathbf{Q} + b\mathbf{P} , \btQ) = a \mathcal{E}( \mathbf{Q}, \btQ) + b \mathcal{E}(\mathbf{P},\btQ).
\end{align}
On the other hand, if we regard the arguments of $\mathcal{E}$ as triples of measures rather than quantile functions, then $\mathcal{E}( \bm{\pi}, \tilde{\bm{\pi}} )$ is linear in its \textit{second} argument. Indeed, note that $\int_0^1 Q_i(\tilde{Q}_i(t)) \, \mathrm{d}t = \int_0^1 Q_i(x) \, \tilde{\pi}_i(\mathrm{d}x)$. It follows that for any triples $\tilde{\bm{\pi}}$ and $\tilde{\bm{\sigma}}$ of probability measures supported on $[0,1]$ and any $\lambda \in [0,1]$, we have 
\begin{equation} \label{eq:linear2}
    \mathcal{E}\big( \bm{\pi} , \, \lambda \tilde{\bm{\pi}} + (1-\lambda) \tilde{\bm{\sigma}} \big) = \lambda \mathcal{E}( \bm{\pi} ,  \tilde{\bm{\pi}} ) + (1-\lambda) \mathcal{E}( \bm{\pi} ,   \tilde{\bm{\sigma}} ).
\end{equation}

Let $\mathbf{t}$ denote the triple of quantile functions $Q_1(t)=Q_2(t)=Q_3(t) = t$.  Setting $f(x) = x$ in \eqref{eq:fx}, the trace equality \eqref{eq:trace1} then reads
\begin{align} \label{eq:trace}
\mathrm{tr}(\mathbf{Q}) := \mathcal{E}(\mathbf{Q},\mathbf{t}) = \mathcal{E}( \mathbf{t},\mathbf{Q}) = \int_0^1 - Q_1(t)-Q_2(t)+Q_3(t) \, \mathrm{d}t = 0.
\end{align}
Given a triple of quantile functions $\mathbf{Q} = (Q_1(t),Q_2(t),Q_3(t))$ and $\mu \ge 0$, we can define a new triple 
\begin{align*}
\mathbf{Q} + \mu \mathbf{t} := \big(Q_1(t)+\mu t,\, Q_2(t)+\mu t, \, Q_3(t)+\mu t\big).
\end{align*}
Recall that $\eta_{\mathbf{Q}} := 1 - \max_{i=1,2,3} \sup_{t\in[0,1]} Q_i(t)$.
If the triple of probability measures corresponding to $\mathbf{Q}$ are supported on $[0,1]$, then so are the measures corresponding to $\mathbf{Q} + \mu \mathbf{t}$ for $0 \le \mu \le \eta_{\mathbf{Q}}$. Note that $\mathrm{tr}(\mathbf{Q} + \mu \mathbf{t}) = \mathrm{tr}(\mathbf{Q}) - \mu/2$.

The precise statement of the self-characterisation theorem is as follows.

\begin{restatable}[The self-characterisation theorem]{thm}{selfcharthm} \label{thm:SC}
Let $\mathbf{Q} = (Q_1,Q_2,Q_3)$ be a triple of integrable quantile functions satisfying $\mathrm{tr}(\mathbf{Q}) = 0$. Then
\begin{align*}
\mathbf{Q} \in \mathscr{H}(\R) &\iff \mathcal{E}(\mathbf{Q},\btQ + \mu \mathbf{t}) \geq 0 \text{ for every $\btQ \in \mathscr{H}[0,1)$ and $0 < \mu \le \eta_{\btQ}$.}
\end{align*}
Moreover, if each $Q_i$ takes values in $[0,1]$, then
\begin{align*}
\mathbf{Q} \in \mathscr{H}[0,1] &\iff \mathcal{E}(\mathbf{Q},\btQ + \mu \mathbf{t}) \geq 0 \text{ for every $\btQ \in \mathscr{H}[0,1]$ and $0 \le \mu \le \eta_{\btQ}$.}
\end{align*}
\end{restatable}

\begin{remark}
        As we explain below, if $\mathbf{Q}$ is a triple of quantile functions of empirical spectral measures of $n$-by-$n$ matrices, and if $\btQ = (Q_{I,n}, Q_{J,n}, Q_{K,n})$ for some triple $(I,J,K) \in T^n_r$, then then corresponding Horn inequality can be expressed as
    \[
    \mathcal{E} \Big( \mathbf{Q}, \btQ + \frac{r}{n} \mathbf{t} \Big) \ge 0.
    \]
    Thus the parameter $\mu$ appearing in the inequalities in Theorem \ref{thm:SC} should be thought of as a limiting value of the ratio $r/n$.
\end{remark}

\begin{remark}
For particular choices of triples, Theorem \ref{thm:SC} recovers various classical inequalities as well as their asymptotic analogues. We have already seen one example above in the limiting Ky Fan inequalities \eqref{eqn:lim-KF}. For another example, observe that for any $a,b \in \R$, the triple $(\delta_a,\delta_b,\delta_{a+b})$ of Dirac masses lies in $\mathscr{H}(\mathbb{R})$, and their respective quantile functions are $\tQ_1(t) = a, \tQ_2(t) = b, \tQ_3(t) = a+b$. The corresponding inequalities tell us that for any $a,b$ with $0 \leq a,b, a+b \leq 1$, each triple $\mathbf{Q}=(Q_1(t),Q_2(t),Q_3(t)) \in \mathscr{H}(\mathbb{R})$ must satisfy
\begin{equation}
-Q_1(a) - Q_2(b) + Q_3(a+b) \geq 0.
\end{equation}
This is an asymptotic analogue of the Weyl inequalities \cite{Weyl1912} and was previously observed by Bercovici and Li, who also used large-$n$ limits of Horn inequalities to derive an infinite-dimensional version of the Freede--Thompson inequalities in a similar fashion \cite{BL01}. 

Alternatively, let $\mu \in [0,1]$ and fix any quantile function $S:[0,1] \to [0,1-\mu]$, and consider the element $(\tQ_1(t),\tQ_2(t),\tQ_3(t)) = (S(t),0,S(t)) \in \mathscr{H}[0,1]$. Then any triple $\mathbf{Q} = (Q_1(t),Q_2(t),Q_3(t))$ must satisfy
\begin{align}
- \int_0^1 Q_1(S(t)+\mu t) \, \mathrm{d}t - \int_0^1 Q_2(\mu t) \, \mathrm{d}t + \int_0^1 Q_3(S(t)+\mu t) \, \mathrm{d}t \geq 0.
\end{align}
This is an asymptotic analogue of the Lidskii--Wielandt inequality \cite{lidskii1950, wielandt1955}.

Naturally, it is possible to obtain asymptotic analogues of various other inequalities by considering different limiting Horn triples.
\end{remark}

\begin{remark}
    Theorem \ref{thm:SC} in fact implies that whenever $\mathbf{Q} \in \mathscr{H}(\R)$ represents a triple of compactly supported probability measures, we have $\mathcal{E}(\mathbf{Q},\btQ + \mu \mathbf{t}) \geq 0$ for every $\btQ \in \mathscr{H}[0,1]$ and $0 \le \mu \le \eta_{\btQ}$. That is, boundedness of each $Q_i$ is enough to guarantee that the inequalities for $\mu = 0$ or $\btQ \in \mathscr{H}[0,1] - \mathscr{H}[0,1)$ also hold in addition to those for $\btQ \in \mathscr{H}[0,1)$ and $0 < \mu \le \eta_{\btQ}$; see Corollary \ref{cor:cpt-supp}. On the other hand, if some measure represented by $\mathbf{Q}$ has unbounded support, then the corresponding quantile function is also unbounded in a neighbourhood of 0 and/or 1.  In that case, for $\btQ \in \mathscr{H}[0,1]$ corresponding to a triple of probability measures that do not all have bounded densities near the endpoints of the unit interval, the integrals in the definition \eqref{eq:energy} may diverge, and thus $\mathcal{E}(\mathbf{Q}, \btQ)$ may not be defined. The requirement that $\btQ \in \mathscr{H}[0,1)$ and $0 < \mu \le \eta_{\btQ}$ guarantees that $\btQ + \mu \mathbf{t}$ represents a triple of measures with densities bounded above by $1/\mu$, which in turn guarantees that $\mathcal{E}(\mathbf{Q},\btQ + \mu \mathbf{t})$ is finite.  This is the only reason that different sets of inequalities appear in the criteria for membership in $\mathscr{H}(\R)$ and $\mathscr{H}[0,1]$ in Theorem \ref{thm:SC}.
\end{remark}

\begin{remark}
Given that triples corresponding to Horn inequalities also \textit{solve} all of the Horn inequalities and are dense in $\mathscr{H}[0,1]$, one might wonder whether at least the second part of Theorem \ref{thm:SC} could be established merely by proving that the functional $\mathcal{E}$ is continuous in some appropriate sense.  We emphasise that this is \textit{not} the case; in fact $\mathcal{E}$ is not continuous in our chosen topology, and even if it were, that would not immediately prove the desired result without further knowledge of the set $\mathscr{H}[0,1]$.  Although the continuity properties of $\mathcal{E}$ are a crucial ingredient in the proof of Theorem \ref{thm:SC}, a more delicate density argument at the level of the Horn inequalities themselves is also required. Continuity properties of the composition functional $\mathcal{E}$ are explored in Section \ref{sec:continuity}.
\end{remark}

\subsection{Redundancy in the Horn inequalities}

Our final main result concerns the redundancy in the Horn inequalities in infinite dimensions.  It is well known that the finite-dimensional Horn inequalities are redundant for $n \ge 5$, while in infinite dimensions, any finite subset of Horn inequalities may be discarded without changing the solutions of the system.  Here we use the preceding results on approximation and self-characterisation to prove something much stronger.

\begin{restatable}{thm}{redundancy} \label{thm:redundancy}
    Choose any sequences $(n_k)_{k \ge 1}$, $(r_k)_{k \ge 1}$ of positive integers such that $(r_k / n_k)_{k \ge 1}$ is a dense subset of $(0,1)$.  Let $\mathbf{Q}$ be a triple of integrable quantile functions satisfying $\mathrm{tr}(\mathbf{Q}) = 0$. Then $\mathbf{Q} \in \mathscr{H}(\mathbb{R})$ if and only if
    \begin{equation} \label{eqn:thin-horn}
    \mathcal{E}\Big(\mathbf{Q}, \mathbf{Q}_{I,J,K,n_k} + \frac{r_k}{n_k} \mathbf{t} \Big) \ge 0 \quad \text{ for all $(I,J,K) \in T^{n_k}_{r_k}$, $\, k \ge 1$.}
    \end{equation}
In other words, the Horn inequalities indexed by triples in the sets $(T^{n_k}_{r_k})_{k \ge 1}$ imply the full system of Horn inequalities indexed by all sets $T^n_r$ for $n \ge 2$ and $1 \le r \le n-1$.
\end{restatable}

This result identifies certain ``small'' subsets of Horn inequalities that are equivalent to the full system in the sense that they imply all of the other inequalities.  It is important to note however that these subsets themselves are still redundant in general; identifying nonredundant, i.e.~minimal, subsets of Horn equalities that imply the full system is a more difficult problem.

\subsection{Further basic properties of the asymptotic and extended asymptotic Horn systems}

We close this section by observing some additional fundamental properties of $\mathscr{H}(\R)$ and $\mathscr{H}[0,1]$.  Some of the facts below will be instrumental to arguments later in the paper; others we record merely to help illustrate the geometry of these objects.

\subsubsection{Dirac masses}
By adding multiples of the $n$-by-$n$ identity matrix, we see that $(\delta_a,\delta_b,\delta_{a+b}) \in \mathscr{H}(\R)$ for any $a,b \in \mathbb{R}$. 

\subsubsection{Dilations and translations} \label{sec:dil-trans}
As noted above, by multiplying a Hermitian matrix equation by a scalar, we see that if $(\pi_1,\pi_2,\pi_3)$ is a triple in $\mathscr{H}[0,1]$ (resp. $\mathscr{H}(\mathbb{R})$), then the pushforward $(D^s_\# \pi_1, D^s_\# \pi_2, D^s_\# \pi_3)$ of these measures under a dilation $D^s(x) = sx$ for $s \in [0,1]$ (resp. $s \in \R$) is another element of $\mathscr{H}[0,1]$ (resp. $\mathscr{H}(\mathbb{R})$). Furthermore, by adding multiples of the identity to a Hermitian matrix equation, we see that writing $T^a:\mathbb{R} \to \mathbb{R}$ for translation $T^a(x) := x+a$, for any $a,b \in \mathbb{R}$ and $(\pi_1,\pi_2,\pi_3) \in \mathscr{H}(\R)$ we have 
$(T^a_\#\pi_1,T^b_\#\pi_2,T^{a+b}_\# \pi_3) \in \mathscr{H}(\R)$ also.

\subsubsection{Exchange of coordinates}
If $(\pi_1,\pi_2,\pi_3)$ is an element of $\mathscr{H}[0,1]$ or $\mathscr{H}(\mathbb{R})$, so is $(\pi_2,\pi_1,\pi_3)$.  In either case, $(\pi_1, D^{-1}_\# \pi_3, D^{-1}_\# \pi_2) \in \mathscr{H}(\R)$. 

\subsubsection{The Sudoku property}

Consider the following matrix of probability measures:
\begin{equation}
\begin{tabular}{ c c c }
 $\pi_{1,1}$ & $\pi_{1,2}$ & $\pi_{1,3}$ \\ 
 $\pi_{2,1}$ & $\pi_{2,2}$ & $\pi_{2,3}$ \\  
 $\pi_{3,1}$ & $\pi_{3,2}$ & $\pi_{3,3}$.    
\end{tabular}
\end{equation}
The Sudoku property states that if the first two rows and all three columns of this matrix represent elements of $\mathscr{H}[0,1]$, then so does the bottom row. This follows from considering matrix equations obtained by adding rows and columns of the $2$-by-$2$ matrix of $n$-by-$n$ matrices $(A_{n,i,j})_{1 \leq i, j \leq 2}$.

\subsubsection{Convexity properties} \label{sec:convexity}

The asymptotic Horn system has two different convexity properties that it inherits from distinct linear structures on two different spaces of measures on the line, as illustrated in \eqref{eq:linear} and \eqref{eq:linear2}: one arises from the usual addition of finite signed measures, and the other arises from pointwise addition of quantile functions of probability measures.

More explicitly, consider triples $\bm{\pi} := (\pi_1,\pi_2,\pi_3)$ and $\bm \pi':= (\pi_1',\pi_2',\pi_3')$ of probability measures with associated triples of quantile functions $\mathbf{Q} = (Q_1,\ldots,Q_3)$ and $\mathbf{Q}' = (Q_1',Q_2',Q_3')$. Our first way of combining triples of probability measures is the \textbf{vertical convex combination}
\begin{align*}
\lambda \bm{\pi} + (1-\lambda) \bm \pi':= \big(\lambda \pi_1 + (1-\lambda)\pi'_1, \, \lambda \pi_2+ (1-\lambda)\pi'_2 , \, \lambda \pi_3 + (1-\lambda) \pi'_3 \big),
\end{align*}
where for each entry in the triple, $\lambda \pi_i + (1-\lambda) \pi_i'$ is simply a linear combination of measures, i.e.\ $(\lambda \pi_i + (1-\lambda) \pi_i')(E) := \lambda \pi_i(E)+ (1-\lambda)\pi_i'(E)$ for Borel subsets $E \subseteq \mathbb{R}$.

Our second way of combining triples of probability measures is the \textbf{horizontal convex combination} via addition of quantile functions, that is,
\begin{align*}
\lambda \mathbf{Q} + (1 - \lambda)\mathbf{Q}' := \big(\lambda Q_1 + (1-\lambda)Q'_1, \, \lambda Q_2 + (1-\lambda)Q'_2, \, \lambda Q_3 + (1-\lambda) Q'_3 \big). 
\end{align*}

Note that both vertical and horizontal convex combinations of triples of measures supported on a subinterval $E \subseteq \mathbb{R}$ are themselves supported on $E$. Moreover, if $\bm \pi $ and $\bm \pi '$ are trace-zero triples, so are $\lambda \bm{\pi} + (1-\lambda) \bm \pi'$ and $\lambda \mathbf{Q} + (1 - \lambda)\mathbf{Q}'$. In fact, we have the following result:

\begin{proposition} \label{prop:convexity}
For any subinterval $E \subseteq \mathbb{R}$, the set $\mathscr{H}(E)$ is closed under both vertical and horizontal convex combinations.
\end{proposition}

\begin{proof}
As noted above, both horizontal and vertical convex combinations of triples of probability measures supported in $E$ are themselves supported in $E$.

To establish closure under vertical convex combinations, one can consider block diagonal matrices consisting of a $\lfloor \lambda n \rfloor$-by-$\lfloor \lambda n \rfloor$ block and a $\lceil (1-\lambda)n \rceil$-by-$\lceil(1-\lambda)n \rceil$ block along the diagonal. 

To establish closure under horizontal convex combinations, first observe that if $x,y \in \R^n$ are vectors with nonincreasing coordinates, and if $Q_x$, $Q_y$ are the quantile functions of their respective empirical measures, then $Q_x + Q_y$ is the quantile function of the empirical measure of $x + y$.  That is, vector addition corresponds to addition of quantile functions.  Now suppose that $\mathbf{Q}, \mathbf{Q}' \in \mathscr{H}(E)$, and choose sequences $\mathbf{Q}_n, \mathbf{Q}'_n \in \mathscr{H}_n(\R)$ such that $\mathbf{Q}_n \to \mathbf{Q}$ and $\mathbf{Q}'_n \to \mathbf{Q}'$.  By Theorem \ref{thm:horn}, $\mathscr{H}_n(\R)$ is determined by linear inequalities on triples of spectra and is therefore closed under convex combinations of these triples regarded as \textit{vectors}; by our preceding observation this means that $\mathscr{H}_n(\R)$ is closed under convex combinations of triples of quantile functions, and we thus have $\lambda \mathbf{Q}_n + (1-\lambda)\mathbf{Q}'_n \in \mathscr{H}_n(\R)$ for $0 \le \lambda \le 1$. Since $\lambda \mathbf{Q}_n + (1-\lambda)\mathbf{Q}'_n \to \lambda \mathbf{Q} + (1-\lambda)\mathbf{Q}'$, we then have $\lambda \mathbf{Q} + (1-\lambda)\mathbf{Q}' \in \mathscr{H}(\R)$, and since all quantile functions in this triple take values in $E$, we in fact have $\lambda \mathbf{Q} + (1-\lambda)\mathbf{Q}' \in \mathscr{H}(E)$.
\end{proof}

We note that with respect to vertical convex combinations, in the special case where the measure $\pi_i$ has support lying to the left of $\tilde{\pi}_i$ (i.e.,~$\pi_i$ is supported on $(-\infty,x]$ and $\tilde{\pi_i}$ is supported on $[x,\infty)$ for some $x \in \R$), a convex combination of measures corresponds to a \textit{concatenation} of quantile functions. That is, the quantile function of $\lambda \pi_i + (1-\lambda)\tilde{\pi}_i$ is given by $Q_i(t/\lambda)$ for $t < \lambda$ and $\tilde{Q}_i((t-\lambda)/(1-\lambda))$ for $t \geq \lambda$. 

Recall that an element of a convex set $K$ is \textbf{extremal} if it cannot be written as a proper convex combination of two distinct elements of $K$. One can verify by using \eqref{eq:linear2} in the setting of Theorem \ref{thm:SC} that in order to test whether some triple $\mathbf{Q}$ belongs to $\mathscr{H}(\R)$ or $\mathscr{H}[0,1]$, one need only consider $\tilde{\mathbf{Q}}$ that are extremal in the vertical sense (i.e.~with respect to addition of the associated probability measures). 
 
\subsubsection{Free and classical probability}

Finally, we remark that the extended asymptotic Horn system $\mathscr{H}(\mathbb{R})$ includes triples of measures that play significant roles in both classical and free probability.

We say that a probability measure $\Pi$ on $\mathbb{R}^2$ is a \textbf{coupling} of $\pi_1$ and $\pi_2$ if for all Borel subsets $A$ of $\mathbb{R}$ we have $\Pi (A \times \mathbb{R}) = \pi_1(A)$ and $\Pi(\mathbb{R} \times A) = \pi_2(A)$. For any probability measures $\pi_1$ and $\pi_2$ with finite expectation,
\begin{align*}
(\pi_1,\pi_2,\pi_3) \in \mathscr{H}(\mathbb{R}) \text{ whenever $\pi_3$ is the law of $X+Y$ under a coupling $\Pi$ of $\pi_1$ and $\pi_2$}.
\end{align*}
This property follows from considering diagonal matrices $(B_{n,i})_{n \geq 1}$ with decreasing entries and empirical spectra converging to $\pi_i$ ($i=1,2$), and then setting $A_{n,1}=B_{n,1}$ and $A_{n,2} = \Sigma_n B_{n,2} (\Sigma_n)^{-1}$ for a suitable sequence of permutation matrices $(\Sigma_n)_{n \geq 1}$.

Next, suppose that $(B_{n,i})_{n \geq 1}$ are sequences of Hermitian matrices with empirical spectra converging to compactly supported measures $\pi_i$ ($i=1,2$). For $i =1,2$, let $(U_{n,i})_{n \geq 1}$ be independent sequences of unitary random matrices distributed according to Haar measure, and define $A_{n,i} := U_{n,i}B_{n,i}U_{n,i}^{-1}$. Then it is known from results of free probability (see e.g. \cite{MingoSpeicher}) that the empirical spectrum of the random matrix sum $A_{n,1}+A_{n,2}$ converges almost surely to a probability measure $\pi_1 \boxplus \pi_2$ known as the additive free convolution of $\pi_1$ and $\pi_2$. In particular, we have
\begin{align*}
(\pi_1,\pi_2,\pi_1 \boxplus \pi_2) \in \mathscr{H}(\mathbb{R}).
\end{align*}
From the free probability perspective, $\pi_1 \boxplus \pi_2$ may be regarded as the ``central'' element of the asymptotic Horn body $\mathscr{H}_{\pi_1,\pi_2}$ defined above in \eqref{eqn:asymp-body-def}.

As a consequence, Theorem \ref{thm:SC} gives many new integral inequalities for free convolutions.  Immediately, for any $\mathbf{Q} = (Q_A,Q_B,Q_C) \in \mathscr{H}[0,1]$ and any $0 \le \mu \le \eta_{\mathbf{Q}}$, we have
\begin{equation} \label{eqn:fc-ineq1}
    \int_0^1 Q_{\pi_1 \boxplus \pi_2}(Q_C(t) + \mu t) \, \mathrm{d}t \ge \int_0^1 Q_{\pi_1}(Q_A(t) + \mu t) \, \mathrm{d}t + \int_0^1 Q_{\pi_2}(Q_B(t) + \mu t)\, \mathrm{d}t.
\end{equation}
Moreover, if $\pi_1,\pi_2,\pi_1 \boxplus \pi_2$ are all supported on $[0,1]$, then $(\pi_1,\pi_2, \pi_1 \boxplus \pi_2) \in \mathscr{H}[0,1]$. In that case, by exchanging the roles of $(Q_A,Q_B,Q_C)$ and $(Q_{\pi_1}, Q_{\pi_2}, Q_{\pi_1 \boxplus \pi_2})$ in \eqref{eqn:fc-ineq1}, we can say more.  For example, if $\pi_1$ is a probability measure supported on $[0,\mu]$ and $\pi_2$ is a probability measure supported on $[0,1-\mu]$ for $0 \le \mu \le 1$, then $\pi_1 \boxplus \pi_2$ is supported on $[0,1]$. For any $\mathbf{Q} = (Q_A,Q_B,Q_C) \in \mathscr{H}(\R)$ we then have
    \begin{equation}
        \int_0^1 Q_C(t) \, (\pi_1 \boxplus \pi_2)(\mathrm{d}t) \ge \int_0^1 Q_A(t) \, \pi_1 (\mathrm{d}t) + \int_0^1 Q_B(t) \, \pi_2 (\mathrm{d}t).
    \end{equation}
The above inequalities hold, for example, when $(Q_A,Q_B,Q_C)$ are the quantile functions of $*$-distributions of self-adjoint operators $A+B=C$ in a finite factor.

\section{The quantile function approach to the Horn inequalities} \label{sec:hornineq}

In this section, we show how the Horn inequalities for $n$-by-$n$ matrices may be written as integral inequalities on the quantile functions of the empirical spectral measures.  We then show that this formulation allows us to state the inequalities in a way that does not depend on $n$, pointing to natural asymptotic questions about how the Horn inequalities and their solutions behave as $n$ grows large.  Such questions may be understood in terms of the geometry and topology of the asymptotic Horn system $\mathscr{H}[0,1]$, and they motivate our study of this object.  We also show that, in the infinite-dimensional setting, the extended asymptotic Horn system $\mathscr{H}(\R)$ is exactly the set of solutions of all Horn inequalities for all finite $n$.

As mentioned above in Section \ref{sec:approx-thm}, 
Bercovici and Li \cite{BL06} previously gave a different $n$-independent formulation of the Horn inequalities, which is closely related to the one that we present here.  Their formulation was used in \cite{BL06,BCDLT} to solve the Horn problem in the infinite-dimensional setting of an arbitrary finite factor.  The key difference between the approaches taken in \cite{BL06} and in the present paper is that here we identify the finite-$n$ Horn inequalities with triples of atomic probability measures rather than triples of subsets of $[0,1]$.  This allows us to give a more symmetrical statement of the Horn inequalities in terms of the functional $\mathcal{E}$, whose arguments are two triples of probability measures (via their quantile functions). It then becomes possible to study the limiting behavior of the Horn inequalities themselves by considering weak or Wasserstein convergence of these measures.

\subsection{Inequalities for quantile functions}

Here we show how the Horn inequalities \eqref{eqn:hornform} can be written in an equivalent form in terms of two triples of quantile functions: the quantile functions of the empirical spectral measures of the matrices $A+B=C$, and the quantile functions $(Q_{I,n}, Q_{J,n}, Q_{K,n})$ associated with each triple $(I,J,K) \in T^n_r$ via \eqref{eq:Iquant}.

First we review Horn's procedure for enumerating the triples $(I,J,K)$ of sets that index the Horn inequalities. We index elements of such subsets in increasing order, writing e.g.~$I = \{i_1 < \hdots < i_r\}$. Now define, for $n \ge 2$ and $1 \le r \le n-1$,
\begin{equation} \label{eqn:Unr-def}
U^n_r = \left\{ (I,J,K) \ \bigg | \ I,J,K \subset \{1, \hdots, n\}, \ |I| = |J| = |K| = r, \ \sum_{i \in I} i + \sum_{j \in J} j = \sum_{k \in K} k + \frac{r(r+1)}{2} \right\}.
\end{equation}
The Horn inequalities correspond to certain subsets $T^n_r \subset U^n_r$. These are defined by setting $T^n_1 = U^n_1$, and for $r > 1$,
\begin{equation} \label{eqn:Tnr-def}
T^n_r = \left\{ (I,J,K) \in U^n_r \ \bigg | \ \textrm{for all } p < r \textrm{ and all } (F,G,H) \in T^r_p, \ \sum_{f \in F} i_f + \sum_{g \in G} j_g \le \sum_{h \in H} k_h + \frac{p(p+1)}{2} \right\}.
\end{equation}

The definition of the sets $T^n_r$ can also be stated in terms of integer partitions $\lambda = (\lambda_1 \ge \hdots \ge \lambda_r)$. 
 For any $I \subset \{1, \hdots, n\}$ with $|I| = r$, we define a corresponding partition $\lambda(I)$ with $r$ parts by
\begin{equation} \label{eqn:lambda-I-def}
\lambda(I) = \big( i_r - r, \, i_{r-1} - (r-1), \, \hdots, \, i_1 - 1 \big).
\end{equation}
Then we have the following alternative characterisation of the sets $T^n_r$, discussed already in the introduction, which makes the recursive structure of the Horn inequalities more apparent (see \cite[Thm.\ 2]{Fu}).

\begin{thm}\label{thm:horn-partitions}
Let $I, J, K \subset \{1, \hdots, n\}$ with $|I| = |J| = |K| = r$. Then $(I,J,K) \in T^n_r$ if and only if there exist $r$-by-$r$ Hermitian matrices $(A,B,A+B)$ with spectra $(\lambda(I), \lambda(J), \lambda(K))$.
\end{thm}

As described above in Section \ref{sec:embed}, triples of index sets $(I,J,K) \in T^n_r$ are in correspondence with certain triples of quantile functions $(Q_{I,n}, Q_{J,n}, Q_{K,n})$. Concretely, for $I \subset \{1,\ldots,n\}$ with $|I| = r$, the empirical measure of $\frac{1}{n}\lambda(I)$ is 
\begin{align} \label{eq:pidef}
\pi_{I,n}(\mathrm{d}x) := \frac{1}{r} \sum_{s=1}^r \delta_{(i_s-s)/n}(\mathrm{d}x).
\end{align}
Note that $\pi_{I,n}$ is supported on $\{0/n,\ldots,(n-r)/n\}$. The quantile function of $\pi_{I,n}$ is 
\begin{align} \label{eq:Iquant2}
Q_{I,n}(t) = \frac{i_s-s}{n}, \qquad t \in \left[ \frac{s-1}{r} , \frac{s}{r} \right),  \quad s \in \{1,2,\ldots,r\}.
\end{align}
Given a triple $(I,J,K)$ of subsets of $\{1,\hdots,n\}$ with the same cardinality, we write
\begin{align} \label{eqn:QIJK-def}
\mathbf{Q}_{I,J,K,n} := (Q_{I,n}(t),Q_{J,n}(t),Q_{K,n}(t)).
\end{align}

\begin{remark} \label{rem:bijection}
The quantile function $Q_{I,n}$ is $n$-integral, $r$-atomic, and takes values in $[0,1-r/n]$. Conversely, by \eqref{eq:Iquant2}, any $n$-integral, $r$-atomic quantile function taking values in $[0,1-r/n]$ gives rise to a subset $I \subset \{1,\ldots,n\}$ with $|I|=r$. Thus the map $(I,J,K) \mapsto \mathbf{Q}_{I,J,K,n}$ is a bijection between triples of subsets of $\{1,\ldots,n\}$ of cardinality $r$ and triples of $n$-integral, $r$-atomic quantile functions supported on $[0,1-r/n]$.
\end{remark}

In the next few lemmas, we will reformulate the definitions of the sets $U_r^n$ and $T_r^n$ in terms of the associated quantile functions.

\begin{lem} \label{lem:Umember}
Let $(I,J,K)$ be a triple of subsets of $\{1,\ldots,n\}$ of cardinality $r$. Then $(I,J,K)$ lie in $U_r^n$ if and only if $\mathrm{tr}(\mathbf{Q}_{I,J,K,n}) = 0$. In particular, $U_r^n$ is in bijection with the $n$-integral and $r$-atomic triples of quantile functions taking values in $[0,1-r/n]$ and satisfying $\mathrm{tr}(\mathbf{Q})=0$. 
\end{lem}

\begin{proof}
We begin by proving that the equation $\sum_{i \in I} i + \sum_{j \in J} j = \sum_{k \in K} k + \frac{r(r+1)}{2}$ is equivalent to $\mathrm{tr}(\mathbf{Q}_{I,J,K,n}) = 0$. To this end, note that
\begin{align*}
\int_0^1 Q_{I,n}(t) \, \mathrm{d}t = \frac{1}{r} \sum_{s=1}^r \frac{ i_s - s}{n} = \frac{1}{nr} \left(  \sum_{i \in I} i  - \frac{r(r+1)}{2} \right).
\end{align*}
Using the definition \eqref{eq:trace} of $\mathrm{tr}(\mathbf{Q}_{I,J,K,n})$ we have
\begin{align*}
\mathrm{tr}(\mathbf{Q}_{I,J,K,n}) = \frac{1}{nr} \left( \frac{r(r+1)}{2} +   \sum_{k \in K} k  - \sum_{i \in I} i - \sum_{j \in J } j \right),
\end{align*}
which is zero if and only if $\sum_{i \in I} i + \sum_{j \in J} j = \sum_{k \in K} k + \frac{r(r+1)}{2}$, as claimed.

By Remark \ref{rem:bijection}, there is a bijection between the triples $(I,J,K)$ of subsets of $\{1,\ldots,n\}$ of cardinality $r$ and the triples of $n$-integral and $r$-atomic quantile functions supported on $[0,1-r/n]$. It thus follows that $U_r^n$ is in bijection with such triples $\mathbf{Q}$ that additionally satisfy $\mathrm{tr}(\mathbf{Q})= 0$, completing the proof.
\end{proof}

Our next lemma characterises membership of $T_r^n$ in terms of quantile functions and the composition functional $\mathcal{E}(\mathbf{Q},\btQ)$ defined in \eqref{eq:energy}.

\begin{lem} \label{lem:Tmember}
Let $(I,J,K)$ be a triple of subsets of $\{1,\ldots,n\}$ of cardinality $r$. Then $(I,J,K) \in T_r^n$ if and only if $\mathrm{tr}(\mathbf{Q}_{I,J,K,n}) = 0$ and
\begin{align} \label{eq:QQtest}
\mathcal{E}\Big( \mathbf{Q}_{I,J,K,n}, \mathbf{Q}_{F,G,H,r} + \frac{p}{r} \mathbf{t} \Big) \geq 0 \qquad \text{for all $(F,G,H) \in T_p^r$, for all $1 \leq p \leq r-1$.}
\end{align} In particular, $T_r^n$ is in bijection with the set of triples of $n$-integral and $r$-atomic quantile functions taking values in $[0,1-r/n]$ and satisfying both $\mathrm{tr}(\mathbf{Q}) = 0$ and \eqref{eq:QQtest}.
\end{lem}

Before proving Lemma \ref{lem:Tmember}, we have the following calculation, which describes integration against composition with the quantile function $Q_{I,n}(t) + \frac{r}{n}t$. 

\begin{lem} \label{lem:fincom}
Let $P$ be an integrable quantile function. Then
\begin{align*}
\int_0^1 P\left( Q_{I,n}(t) + \frac{r}{n}t\right) \mathrm{d}t = \frac{n}{r} \sum_{i \in I} \int_{(i-1)/n}^{i/n} P( t) \, \mathrm{d}t.
\end{align*}
\end{lem}
\begin{proof}
Write $I = \{i_1 < \ldots < i_r\}$.
Using \eqref{eq:Iquant} to obtain the second equality below, we have
\begin{align*}
\int_0^1 P\left( Q_{I,n}(t) + \frac{r}{n}t \right) \mathrm{d}t &= \sum_{ s = 1}^r \int_{(s-1)/r}^{s/r}  P\left( Q_{I,n}(t) + \frac{r}{n}t \right) \mathrm{d}t\\
&= \sum_{s =1}^r \int_{(s-1)/r}^{s/r} P \left( \frac{i_s-s}{n} + \frac{r}{n} t \right) \mathrm{d}t
= \frac{n}{r} \sum_{s=1}^r \int_{(i_s-1)/n}^{i_s/n} P(t) \, \mathrm{d}t,
\end{align*}as required.
\end{proof}

We now complete the proof of Lemma \ref{lem:Tmember}.

\begin{proof}[Proof of Lemma \ref{lem:Tmember}]
We compose $Q_{I,n}(t)$ with $Q_{F,r}(t) + \frac{p}{r} t$. Using Lemma \ref{lem:fincom} to obtain the first equality below, and \eqref{eq:Iquant} to obtain the second, we have 
\begin{align} \label{eq:ocean}
\int_0^1 Q_{I,n} \left( Q_{F,r}(t) + \frac{p}{r} t \right) \mathrm{d}t = \frac{r}{p} \sum_{f \in F} \int_{(f-1)/r}^{f/r} Q_{I,n}(t) \, \mathrm{d}t = \frac{1}{p} \sum_{ f \in F} \frac{ i_f - f }{ n}. 
\end{align}
Then using \eqref{eq:ocean} to obtain the first equality below, and the fact that $(F,G,H) \in U_p^r$ entails $\sum_{f \in F} f + \sum_{g \in G} g = \sum_{h \in H} h + \frac{p(p+1)}{2}$ to obtain the second, we have 
\begin{align} \label{eq:ocean2}
\mathcal{E}\Big( \mathbf{Q}_{I,J,K,n} , \mathbf{Q}_{F,G,H,r} + \frac{p}{r} \mathbf{t} \Big) &=\frac{1}{pn} \left( - \sum_{f \in F} (i_f - f)-  \sum_{g \in G} (j_g - g)  + \sum_{h \in H} (k_h - h) \right) \nonumber \\
& =\frac{1}{pn} \left( - \sum_{f \in F} i_f - \sum_{g \in G} j_g + \sum_{h \in H} k_h + p(p+1)/2 \right).
\end{align}
Now by the definition of $T_r^n$, $(I,J,K)$ belongs to $T_r^n$ if and only if $- \sum_{f \in F} i_f - \sum_{g \in G} j_g + \sum_{h \in H} k_h + p(p+1)/2 \geq 0$ for all relevant $(F,G,H)$. That completes the proof. 
\end{proof}

Finally, we obtain the following proposition, which encodes Theorem \ref{thm:horn} in terms of the quantile functions of the empirical spectral measures:
\begin{prop} \label{prop:hornexp}
Let $\mathbf{Q} = (Q_1(t),Q_2(t),Q_3(t))$ be a triple of $n$-atomic quantile functions. Then $\mathbf{Q}$ are the quantile functions of the empirical spectra of $n$-by-$n$ Hermitian matrices $A+B=C$ if and only if $\mathrm{tr}(\mathbf{Q}) = 0$ and $\mathcal{E}(\mathbf{Q},\mathbf{Q}_{I,J,K,n} + \frac{r}{n} \mathbf{t}) \geq 0$ for all $(I,J,K) \in T_r^n$, $1 \le r \le n-1$. 
\end{prop}
\begin{proof}
For $j = 1, \hdots, n$, define $\alpha_{n+1-j}$,  $\beta_{n+1-j}$, $\gamma_{n+1-j}$ to be the respective constant values of the functions $Q_1$, $Q_2$, $Q_3$ on the interval $[(j-1)/n, j/n)$. The coordinates of $\alpha,\beta,\gamma$ are nonincreasing since the quantile functions $Q_i$ are nondecreasing. Now define
\begin{align} \label{eq:pintor}
\hat{\alpha}_i :=- \alpha_{n+1-i},\quad \hat{\beta}_i := - \beta_{n+1-i}, \quad \text{and} \quad \hat{\gamma_i} := - \gamma_{n+1-i}.
\end{align}
The coordinates of $\hat{\alpha},\hat{\beta},\hat{\gamma}$ are also nonincreasing, and $\alpha,\beta,\gamma$ are the spectra of a Hermitian triple $A + B = C$ if and only if $\hat{\alpha},\hat{\beta},\hat{\gamma}$ are the spectra of the Hermitian triple $(-A) + (-B) = -C$. 

Now note that for $(I,J,K) \in T_r^n$, by Lemma \ref{lem:fincom} we have 
\begin{align*}
\frac{r}{n} \mathcal{E}\Big(\mathbf{Q},\mathbf{Q}_{I,J,K,n} + \frac{r}{n} \mathbf{t}\Big) &= - \sum_{i \in I} \int_{(i-1)/n}^{i/n} Q_1(t) \, \mathrm{d}t - \sum_{j \in J} \int_{(j-1)/n}^{j/n} Q_2(t) \, \mathrm{d}t  +  \sum_{k \in K} \int_{(k-1)/n}^{k/n} Q_3(t) \, \mathrm{d}t \nonumber \\
&= - \sum_{i \in I} \alpha_{n+1-i} - \sum_{j \in J} \beta_{n+1-j} + \sum_{k \in K} \gamma_{n+1-k}  \nonumber \\
&= \sum_{i \in I} \hat{\alpha}_i + \sum_{j \in J } \hat{\beta}_j - \sum_{k \in K} \hat{\gamma}_k. 
\end{align*}
In particular, we have $\mathcal{E}(\mathbf{Q},\mathbf{Q}_{I,J,K,n} + \frac{r}{n} \mathbf{t}) \geq 0$ for all $(I,J,K) \in T_r^n$ if and only if $ \sum_{k \in K} \hat{\gamma}_k \leq  \sum_{i \in I} \hat{\alpha}_i + \sum_{j \in J } \hat{\beta}_j$ for all $(I,J,K) \in T_r^n$, which by Theorem \ref{thm:horn} occurs if and only if $\hat{\alpha},\hat{\beta},\hat{\gamma}$ are the spectra of a Hermitian triple. As argued above, the same must then hold for $\alpha, \beta, \gamma$, and since $\mathbf{Q}$ are the quantile functions corresponding to these three spectra, this completes the proof.
\end{proof}


\subsection{An $n$-independent formulation of the Horn inequalities}

In Proposition \ref{prop:hornexp} we have stated the Horn inequalities \textit{for each fixed} $n$ in terms of the quantile functions of empirical spectral measures.  We now upgrade this to a uniform statement giving a single collection of infinitely many inequalities that characterises the solutions for all $n$ simultaneously.

The probability measure $\pi_{I,n}$ is obtained by considering $I$ as a subset of $\{1,\ldots,n\}$. Note that if $I \subset \{1,\ldots,n\}$, and $n' > n$, we may also think of $I$ as a subset of $\{1,\ldots,n'\}$. The relationship between the probability measures $\pi_{I,n}$ and $\pi_{I,n'}$ can be expressed in terms of their quantile functions; namely, by \eqref{eq:Iquant} we have
\begin{align} \label{eq:quantrel}
Q_{I,n'}(t) = \frac{n}{n'} Q_{I,n}(t).
\end{align}
Since $\mathcal{E}(\mathbf{Q},\btQ)$ is linear in its first argument, it follows in particular that for any sets $(I,J,K)$,
\begin{align*}
\mathcal{E}( \mathbf{Q}_{I,J,K,n'} , \mathbf{Q}_{F,G,H,r} ) = \frac{n}{n'}\mathcal{E}( \mathbf{Q}_{I,J,K,n},\mathbf{Q}_{F,G,H,r} ),
\end{align*}
which, by virtue of Lemma \ref{lem:Tmember}, in turn implies
\begin{align*}
(I,J,K) \in T_r^n \implies (I,J,K) \in T_r^{n'} \qquad \forall \, n' \geq n.
\end{align*}

We can deduce a further, related statement by considering dilations of subsets. Let $m \geq 1$ be an integer. With every subset $I \subset \{1,\ldots,n\}$ of cardinality $r$ we may associate a subset $mI \subset \{1,\ldots,mn\}$ of cardinality $mr$ by setting
\begin{align*}
mI = \bigcup_{i \in I} \{ mi - m +1, mi - m+2,\ldots,mi \}.
\end{align*}
A brief calculation tells us that with $\pi_{I,n}(\mathrm{d}x)$ as in \eqref{eq:pidef} we have
\begin{align*}
\pi_{mI,mn} = \pi_{I,n}, \qquad \text{ or equivalently,} \qquad Q_{mI,mn} = Q_{I,n}.
\end{align*}
It follows in particular from Lemma \ref{lem:Umember} that
\begin{align*}
\{ \mathbf{Q}_{I,J,K,n} \ | \ (I,J,K) \in U_{r}^{n} \} \subseteq \{ \mathbf{Q}_{I,J,K,mn} \ | \ (I,J,K) \in U_{mr}^{mn} \}.
\end{align*}
In fact, we have the following lemma.
\begin{lem} \label{lem:dilation}
Let $(I,J,K)$ be a triple of subsets of $\{1,\ldots,n\}$ of the same cardinality $r$. Then
\begin{align*}
(I,J,K) \in T_r^n \iff (mI,mJ,mK) \in T_{mr}^{mn} \qquad \text{for every integer $m \ge 1$}.
\end{align*}
\end{lem}
\begin{proof}
That the latter implies the former follows from setting $m=1$. On the other hand, if $(I,J,K) \in T_r^n$, then by Theorem \ref{thm:horn-partitions}, there are Hermitian matrices $A+B=C$ with eigenvalues given by the partitions $\lambda(I),\lambda(J),\lambda(K)$, or equivalently, with empirical spectra given by $\pi_{I,n},\pi_{J,n},\pi_{K,n}$ as in \eqref{eq:pidef}. By considering $mn$-by-$mn$ block diagonal matrices with $m$ copies of the $n$-by-$n$ matrices $A$, $B$, $C$ along the diagonal, it follows that there exist $mn$-by-$mn$ Hermitian matrices with these same empirical spectral measures. Then since $(\pi_{I,n},\pi_{J,n},\pi_{K,n}) = (\pi_{mI,mn},\pi_{mJ,mn},\pi_{mK,mn})$, it follows from the other direction of Theorem \ref{thm:horn-partitions} that $(mI,mJ,mK) \in T_{mr}^{mn}$, as required.
\end{proof}

We now use the dilation property to prove the following \textit{forward} property, which says that if a triple of quantile functions is constant on intervals of length $1/n$ and satisfies the Horn inequalities corresponding to all $(I,J,K) \in T_r^n$, then it satisfies the Horn inequalities corresponding to all $(I,J,K)$ in $T_{r'}^{n'}$ for all integers $r',n'$ with $1 \leq r' \leq n'-1$ (c.f.\ the statement of Proposition \ref{prop:hornexp}). This is the desired $n$-independent formulation of the Horn inequalities: it shows that, for any value of $n$, a necessary and sufficient condition to be the spectra of a Hermitian triple is given by the full set of quantile function inequalities corresponding to \textit{all} triples $(I,J,K) \in T^{n'}_{r'}$ for all values of $n'$ and $r'$.

\begin{prop} \label{prop:forward}
Let $\mathbf{Q}$ be a triple of $n$-atomic quantile functions satisfying $\tr(\mathbf{Q})=0$. The following three conditions are equivalent:
\begin{enumerate}
\item $\mathbf{Q}$ belongs to $\mathscr{H}_n(\mathbb{R})$.
\item $\mathcal{E}\left(\mathbf{Q},\mathbf{Q}_{I,J,K,n} + \frac{r}{n}\mathbf{t} \right) \geq 0$ for all $(I,J,K)$ in all $T_r^n$, $1 \leq r \leq n-1$.
\item $\mathcal{E}\left(\mathbf{Q},\mathbf{Q}_{I,J,K,n'} + \frac{r'}{n'}\mathbf{t} \right) \geq 0$ for all $(I,J,K)$ in all $T_{r'}^{n'}$, where $r',n'$ are any integers $1 \leq r' \leq n'-1$.
\end{enumerate}
\end{prop}

\begin{proof}
The equivalence of (1) and (2) is Proposition \ref{prop:hornexp}. It is clear that (3) implies (2). It thus suffices to show that (2) implies (3). To this end, fix some $(I,J,K) \in T_{r'}^{n'}$ for an arbitrary choice of integers $1 \le r' \le n'-1$. Assuming (2), by Proposition \ref{prop:hornexp} there are $n$-by-$n$ Hermitian matrices $A+B=C$ with spectral quantiles given by $\mathbf{Q}$. By considering block diagonal matrices (with $n'$ copies of the same $n$-by-$n$ matrix along the diagonal), it is then easily verified that there are $n'n$-by-$n'n$ Hermitian matrices $A'+B'=C'$ with these same spectral quantiles $\mathbf{Q}$. In particular, now using the other direction of Proposition \ref{prop:hornexp}, it follows that $\mathcal{E}\left(\mathbf{Q},\mathbf{Q}_{F,G,H,n'n} + \frac{s}{n'n}\mathbf{t} \right) \geq 0$ for any $(F,G,H) \in T_s^{n'n}$ with $1 \leq s \leq n'n-1$. 

Now by Lemma \ref{lem:dilation}, $(I,J,K) \in T_{r'}^{n'}$ implies $(nI,nJ,nK) \in T_{r'n}^{n'n}$. Since there is an $n'n$-by-$n'n$ Hermitian triple with spectral quantiles $\mathbf{Q}$, it follows that
\begin{align} \label{eq:sca1}
\mathcal{E} \left( \mathbf{Q}, \mathbf{Q}_{nI,nJ,nK,nn'} + \frac{r'n}{n'n} \mathbf{t} \right) \geq 0.
\end{align}
However, since $\pi_{nI,nn'} = \pi_{I,n'}$, and likewise for $nJ$ and $nK$, it follows that
$$\mathbf{Q}_{nI,nJ,nK,nn'} + \frac{r'n}{n'n} \mathbf{t} = \mathbf{Q}_{I,J,K,n'} + \frac{r'}{n'}\mathbf{t}.$$
In particular, we have $\mathcal{E} \left( \mathbf{Q}, \mathbf{Q}_{I,J,K,n'} + \frac{r'}{n'}\mathbf{t}\right) \geq 0$, as required. 
\end{proof}

\subsection{The solution set of the Horn inequalities and the self-averaging property} \label{sec:selfavg}

To conclude this section, we show that $\mathscr{H}(\R)$ is precisely the set of triples of integrable quantile functions that satisfy the trace equality together with all finite-$n$ Horn inequalities indexed by all of the sets $T^n_r$.

\begin{prop} \label{prop:H-ineqs}
Let $\mathbf{Q}$ be a triple of integrable quantile functions. Then $\mathbf{Q} \in \mathscr{H}(\R)$ if and only if $\tr(\mathbf{Q})=0$ and $\mathcal{E}\left(\mathbf{Q},\mathbf{Q}_{I,J,K,n} + \frac{r}{n}\mathbf{t} \right) \geq 0$ for all $(I,J,K) \in T_r^n$, for all $n \ge 2$ and $1 \leq r \leq n-1$.
\end{prop}

The proof of Proposition \ref{prop:H-ineqs} relies on the following ``self-averaging'' property of solutions of the Horn inequalities.  This property is not difficult to show, and an equivalent fact was already remarked in passing in \cite{BL06}, but we believe that it is interesting enough to merit a distinct statement and proof.

\begin{restatable}[The self-averaging property]{prop}{selfavg} \label{prop:selfavg}
Let $\mathbf{Q}$ be a triple of integrable  
quantile functions with $\tr(\mathbf{Q}) = 0$, such that $\mathcal{E}\left(\mathbf{Q},\mathbf{Q}_{I,J,K,n} + \frac{r}{n}\mathbf{t} \right) \geq 0$ for all $(I,J,K) \in T_r^n$, for all $n \ge 2$ and $1 \leq r \leq n-1$.
Define $\mathbf{Q}^n= (Q^n_1(t),Q^n_2(t),Q^n_3(t))$ by averaging $\mathbf{Q}$ over each interval of the form $\left[ \frac{j-1}{n}, \frac{j}{n} \right)$, that is,
\begin{align} \label{eq:quantflat}
Q_i^n(t) := n \int_{(j-1)/n}^{j/n} Q_i(u) \, \mathrm{d}u \qquad \textrm{for }t \in \left[ \frac{j-1}{n},\frac{j}{n} \right), \quad 1 \le j \le n.
\end{align}
Then there exist $n$-by-$n$ Hermitian matrices $A_1+A_2=A_3$ with spectral quantiles given by $\mathbf{Q}^n$. 
\end{restatable}

\begin{proof}
    We first point out that each function $Q_i^n$ is indeed a valid quantile function of a probability measure with bounded support.  It is right-continuous by construction, it is nonincreasing because $Q_i$ is nonincreasing, and it is bounded because $Q_i \in L^1$.  Moreover $\tr(\mathbf{Q}^n) = \tr(\mathbf{Q}) = 0$.

    We now claim that for each $n \geq 1$, $\mathbf{Q}^n$ are the spectral quantile functions of a triple of $n$-by-$n$ Hermitian matrices $A_1 + A_2 = A_3$. To see this, by Proposition \ref{prop:hornexp} we need to verify that $\mathcal{E}( \mathbf{Q}^n , \mathbf{Q}_{I,J,K,n} + \frac{r}{n} \mathbf{t}) \geq 0$ for all $(I,J,K) \in T_r^n$, $1 \le r \le n-1$. However, we actually claim that for all $(I,J,K) \in T_r^n$, we have the equality
\begin{align} \label{eq:neq}
\mathcal{E} \Big( \mathbf{Q}^n , \mathbf{Q}_{I,J,K,n} + \frac{r}{n}\mathbf{t} \Big) = \mathcal{E} \Big( \mathbf{Q} , \mathbf{Q}_{I,J,K,n} + \frac{r}{n} \mathbf{t} \Big).
\end{align}
To see that \eqref{eq:neq} holds, first note that the definition \eqref{eq:quantflat} of $Q_i^n$ entails
\begin{align} \label{eq:reaverage}
\int_{(j-1)/n}^{j/n} Q_i^n(t) \, \mathrm{d}t = \int_{(j-1)/n}^{j/n} Q_i(t) \, \mathrm{d}t
\end{align}
for all $1 \leq j \leq n$. Now using Lemma \ref{lem:fincom} to obtain the first and third equalities below, and \eqref{eq:reaverage} to obtain the second equality, for any subset $I \subset \{1,\ldots,n\}$ of cardinality $r$ we have  
\begin{align} \label{eq:sameish2}
\int_0^1 Q_i^n \left( Q_{I,n}(t) + \frac{r}{n} t \right) \mathrm{d}t &= \frac{n}{r} \sum_{i \in I} \int_{(i-1)/n}^{i/n} Q_i^n(t) \, \mathrm{d}t \nonumber \\
& = \frac{n}{r} \sum_{i \in I} \int_{(i-1)/n}^{i/n} Q_i(t) \, \mathrm{d}t \nonumber \\
&= \int_0^1 Q_i \left( Q_{I,n}(t) + \frac{r}{n} t \right) \mathrm{d}t.
\end{align}
Referring to the definition \eqref{eq:energy} of $\mathcal{E}$, \eqref{eq:neq} follows from \eqref{eq:sameish2}. It then follows from the assumptions on $\mathbf{Q}$ that $\mathcal{E}( \mathbf{Q}^n , \mathbf{Q}_{I,J,K,n} + \frac{r}{n} \mathbf{t}) \geq 0$ for all $(I,J,K) \in T_r^n$, which by Proposition \ref{prop:hornexp} guarantees that $\mathbf{Q}^n$ are the spectral quantiles of $n$-by-$n$ Hermitian matrices $A_1+A_2=A_3$.
\end{proof}

Once we have proved Proposition \ref{prop:H-ineqs}, we can conclude that the conditions on $\mathbf{Q}$ in Proposition \ref{prop:selfavg} are equivalent to assuming $\mathbf{Q} \in \mathscr{H}(\mathbb{R})$. Thus for each $n$, the averaging map is a linear projection on $L^1([0,1])$, and the above proposition says that it induces a surjection $\mathscr{H}(\mathbb{R}) \twoheadrightarrow \mathscr{H}_n(\mathbb{R})$. In particular, for each $m$, the operation of taking the $n$-average of $m$-atomic quantile functions extends to a linear transformation $\R^{3m} \to \R^{3n}$ that maps spectra of $m$-by-$m$ Hermitian triples to spectra of $n$-by-$n$ Hermitian triples.  Note that averaging maps for different values of $n$ do not generally commute.

The elementary lemma below shows that $\mathbf{Q}^n \to \mathbf{Q}$ as $n \to \infty$. It thus follows that for any triple $\mathbf{Q} \in \mathscr{H}(\mathbb{R})$, Proposition \ref{prop:selfavg} gives a concrete construction, for each $n$, of a triple of spectra of $n$-by-$n$ matrices satisfying $A_n+B_n=C_n$, and whose spectral quantile functions converge to $\mathbf{Q}$ as $n \to \infty$.

\begin{lem} \label{lem:avg-L1}
If $Q^n$ is the $n$-atomic average of a quantile function $Q:[0,1] \to [0,1]$, then 
\begin{align} \label{eq:01case}
||Q^n - Q ||_1 = \int_0^1 |Q^n(t) -Q(t)| \, \mathrm{d}t \leq \frac{1}{n}.
\end{align}
More generally, if $Q^n$ is the $n$-atomic average of an integrable quantile function $Q:[0,1] \to \mathbb{R} \cup \{\pm \infty\}$, then
\begin{align} \label{eq:l1case}
||Q^n - Q ||_1 \to 0 \quad \text{as $n \to \infty$}.
\end{align}
\end{lem}

\begin{proof}
First we prove \eqref{eq:01case}.
Let $a_j := Q(j/n) - Q((j-1)/n)$. Then $\sum_{j=1}^n a_j \leq 1$. Moreover, $$\int_{(j-1)/n}^{j/n} | Q^n(u)-Q(u)| \, \mathrm{d}u \leq \frac{a_j}{n}.$$ The claim follows by summing over $j$.

To prove \eqref{eq:l1case}, we fix $K > 0$ and let $J_K$ be the set of $j \in \{1, \hdots, n\}$ such that $|Q(t)| < K$ for all $t \in [(j-1)/n, j/n)$.
We then write
\begin{align*}
    ||Q^n - Q||_1 = S_1 + S_2,
\end{align*}
with
\begin{align*}
    S_1 = \sum_{j \in J_K} \int_{(j-1)/n}^{j/n} |Q^n(t) -Q(t)| \, \mathrm{d}t, \qquad S_2 = \sum_{j \not \in J_K} \int_{(j-1)/n}^{j/n} |Q^n(t) -Q(t)| \, \mathrm{d}t.
\end{align*}

Reasoning as in the proof of \eqref{eq:01case} above, we find $S_1 < 2K/n$. To control $S_2$, observe that
\[
S_2 = \int_E |Q^n(t) -Q(t)| \, \mathrm{d}t,
\]
where $E := \bigcup_{j \not \in J_K} [(j-1)/n, j/n)$. Set
\begin{align*}
    a_K = \inf \big\{x \in [0,1] \ \big | \ |Q(t)| < K \big \} \vee 0, \qquad b_K = \sup \big \{x \in [0,1] \ \big | \ |Q(t)| < K \big \} \wedge 1.
\end{align*}
Since $Q$ is nondecreasing, we have $E \subseteq [0, a_K + 1/n] \cup [b_K - 1/n, 1]$.  By Markov's inequality, the Lebesgue measure of this pair of intervals, and thus the Lebesgue measure of $E$, is at most $||Q||_1 / K + 2/n$.  Choosing $K$ sufficiently large and sending $n \to \infty$, we then find that $S_1 \to 0$ while $\limsup_{n \to \infty} S_2$ is bounded above by a constant that can be made arbitrarily small.  This completes the proof of \eqref{eq:l1case}.
\end{proof}

Now we return to prove Proposition \ref{prop:H-ineqs}.

\begin{proof}[Proof of Proposition \ref{prop:H-ineqs}]
Let $\mathscr{G}$ be the set of triples $\mathbf{Q}$ of integrable quantile functions such that $\tr(\mathbf{Q})=0$ and $\mathcal{E}\left(\mathbf{Q},\mathbf{Q}_{I,J,K,n} + \frac{r}{n}\mathbf{t} \right) \geq 0$ for all $(I,J,K) \in T_r^n$, $n \ge 2$ and $1 \leq r \leq n-1$.  We will show that $\mathscr{G} = \mathscr{H}(\R)$.

First we observe that $\mathscr{G}$ is a closed subset of $L^1([0,1])^3$, because it is defined as an intersection of closed sets: the set of triples of (almost-everywhere equivalence classes of) quantile functions $\mathbf{Q}$ satisfying the linear equation $\tr{\mathbf{Q}} = 0$ is closed, as is the set of triples satisfying each linear inequality $\mathcal{E}\left(\mathbf{Q},\mathbf{Q}_{I,J,K,n} + \frac{r}{n}\mathbf{t} \right) \geq 0$ for any given $(I,J,K) \in T_r^n$.

Now suppose $\mathbf{Q} \in \mathscr{H}(R)$.  Then there are sequences of $n$-by-$n$ Hermitian matrices $A_n + B_n = C_n$ such that their spectral quantiles $\mathbf{Q}_{ABC,n}$ converge to $\mathbf{Q}$ as $n \to \infty$.  By Proposition \ref{prop:forward}, each $\mathbf{Q}_{ABC,n}$ belongs to $\mathscr{G}$. Thus $\mathbf{Q}$ is a limit point of the closed set $\mathscr{G}$, so $\mathbf{Q} \in \mathscr{G}$, and therefore $\mathscr{H}(\R) \subseteq \mathscr{G}$.

For the other direction, suppose $\mathbf{Q} \in \mathscr{G}$.  By Proposition \ref{prop:selfavg}, the $n$-atomic average $\mathbf{Q}^n$ corresponds to the spectral quantiles of a triple of Hermitian matrices $A_n + B_n = C_n$, and by Lemma \ref{lem:avg-L1}, $\mathbf{Q}^n$ converges to $\mathbf{Q}$.  Since $\mathscr{H}(\R)$ is the set of limits of spectral quantiles of Hermitian triples, we then have $\mathbf{Q} \in \mathscr{H}(\R)$, and thus $\mathscr{G} \subseteq \mathscr{H}(\R)$.
\end{proof}

Finally, we can now prove Lemma \ref{lem:embed}, the embedding lemma identifying spectra of $n$-by-$n$ Hermitian triples with $n$-atomic elements of $\mathscr{H}(\R)$, and $T^n_r$ with $n$-integral, $r$-atomic elements of $\mathscr{H}[0,1-r/n]$.

\begin{proof}[Proof of Lemma \ref{lem:embed}]
For the first statement, suppose that $\mathbf{Q}$ is an $n$-atomic triple of quantile functions in $\mathscr{H}(\mathbb{R})$. Then by Proposition \ref{prop:H-ineqs}, $\mathcal{E}(\mathbf{Q},\mathbf{Q}_{I,J,K,n} + \frac{r}{n}t) \geq 0$ for every $(I,J,K) \in T_r^n$ and every $1 \leq r \leq n-1$. By Proposition \ref{prop:forward}, this implies that $\mathbf{Q}$ belongs to $\mathscr{H}_n(\mathbb{R})$.

For the second statement, if $\mathbf{Q}$ is an $n$-integral and $r$-atomic triple of quantile functions taking values in $[0,1-r/n]$, then by the same logic we have $\mathbf{Q} \in \mathscr{H}[0,1-r/n]$ if and only if $\tr(\mathbf{Q})=0$ and $\mathcal{E}(\mathbf{Q},\mathbf{Q}_{F,G,H,r} + \frac{p}{r}t) \geq 0$ for every $(F,G,H) \in T_p^r$, $1 \le p \le r-1$. Thus the claim to be shown is exactly Lemma \ref{lem:Tmember}.
\end{proof}

\section{Proofs of the main results}
\label{sec:main-proofs}

We now turn to proving our main results.  First we study the continuity properties of the composition functional $\mathcal{E}$.  Although $\mathcal{E}$ is not actually continuous in either argument in our chosen topology, it turns out to have enough continuity for our purposes: it is continuous in its first argument assuming that the second argument is a triple of quantile functions of measures with bounded densities, while in its second argument it is continuous along sequences of triples with uniformly bounded densities. We then prove Theorem \ref{thm:new23}, a key ingredient in the proofs that follow, which gives a quantitative bound on the distance between points of $\mathscr{H}[0,1]$ and triples representing Horn inequalities in specific sets $T^n_r$. With this bound in hand, we prove our three main theorems: Theorem \ref{thm:dense} on approximation by Horn inequalities with fixed asymptotic ratio $r_n / n \to q$, followed by Theorem \ref{thm:SC} on self-characterisation, and finally Theorem \ref{thm:redundancy} on redundancy in the infinite system of Horn inequalities.

\subsection{Continuity properties of the composition functional} \label{sec:continuity}

Recall that a sequence of probability measures $(\pi^n)_{n \geq 1}$ on $\mathbb{R}$ are said to converge weakly to $\pi$ if for all bounded and continuous functions $f:\mathbb{R} \to \R$ we have
\begin{align*}
\int_{-\infty}^\infty f(x) \, \pi^n(\mathrm{d}x) \to \int_{-\infty}^\infty f(x) \, \pi(\mathrm{d}x).
\end{align*}
In this case we write $\pi^n \Longrightarrow \pi$. Note that if $(Q^n)_{n \geq 1}$ and $Q$ are the associated quantile functions, then by \eqref{eq:fx}, weak convergence is equivalent to 
\begin{align} \label{eq:Qweak}
\int_0^1 f(Q^n(t)) \, \mathrm{d}t \to \int_0^1 f(Q(t)) \, \mathrm{d}t 
\end{align}
for all bounded, continuous $f: \R \to \R$. Another equivalent definition is that $\pi^n \Longrightarrow \pi$ if and only if $Q^n(t) \to Q(t)$ at every point of continuity of $Q(t)$ \cite[Prop.\ 5, p.\ 250]{FG}.

Recall further that we have topologised the space of finite-mean probability measures using the Wasserstein distance, and that Wasserstein convergence is equivalent to weak convergence plus convergence in expectation, or to $L^1$ convergence of quantile functions.  For measures with uniformly bounded support, weak convergence implies convergence in Wasserstein distance.

Given triples $\mathbf{Q}^n = (Q^n_1(t),Q^n_2(t),Q^n_3(t))$ and $\mathbf{Q} = (Q_1(t),Q_2(t),Q_3(t))$ of quantile functions, we say that $\mathbf{Q}^n$ converges weakly (resp. in Wasserstein distance) to $\mathbf{Q}$ if each $Q_i^n$ converges to $Q_i$ pointwise at points of continuity of $Q_i$ (resp. in $L^1$).

Suppose that a probability measure $\pi$ has a density with respect to Lebesgue measure that is bounded above by $L > 0$. It is not hard to see that this is equivalent to the property that its quantile function satisfies $Q(s+t) - Q(t) \geq s/L$ for all $0 \leq t \leq s+t \leq 1$.

\begin{lem} \label{lem:continuity}
Let $\mathbf{P}$ be a triple of integrable quantile functions. Let $(\mathbf{Q}^n)_{n \geq 1}$ be a sequence of triples of quantile functions associated with triples $(\bm{\pi}^n)_{n \geq 1}$ of probability measures supported on $[0,1]$ with density functions uniformly bounded above by some $K > 0$, and suppose that $\bm{\pi}^n$ converges weakly to $\bm{\pi}$. Then 
\begin{align} \label{eq:latterconv}
\lim_{n \to \infty} \mathcal{E}(\mathbf{P},\mathbf{Q}^n) = \mathcal{E}(\mathbf{P},\mathbf{Q}).
\end{align}
Alternatively, let $\mathbf{P}$ be a triple of integrable quantile functions, and let $\mathbf{Q}$ be a triple of quantile functions of probability measures supported on $[0,1]$ with bounded densities. Let $\mathbf{P}^n$ be a sequence of triples converging to $\mathbf{P}$ in Wasserstein distance. Then
\begin{align} \label{eq:formerconv}
\lim_{n \to \infty} \mathcal{E}(\mathbf{P}^n,\mathbf{Q}) = \mathcal{E}(\mathbf{P},\mathbf{Q}).
\end{align}
\end{lem}

\begin{proof}
Write $\mathcal{E}(\mathbf{P},\mathbf{Q}) := -E(P_1,Q_1)-E(P_2,Q_2)+E(P_3,Q_3)$, where, for quantile functions $P,Q$ with $Q$ taking values in $[0,1]$, we write
\begin{align} \label{eqn:E-single}
E(P,Q) := \int_0^1 P(Q(t)) \, \mathrm{d}t.
\end{align}
It is then sufficient to study the continuity properties of $E(P,Q)$.

First we prove \eqref{eq:latterconv}. It suffices to show that whenever $Q_n(t)$ is a sequence of quantile functions satisfying $Q_n(t+s) - Q_n(t) \geq s/K$ for all $0 \leq t \leq s+t \leq 1$, and such that $Q_n(t) \to Q(t)$ at every continuity point $t$ of $Q(t)$, we have $E(P,Q_n) \to E(P,Q)$. 

Since the continuous and bounded functions are dense in $L^1([0,1])$, for any $\varepsilon > 0$ we may choose a continuous and bounded $f:[0,1] \to \mathbb{R}$ such that $\int_0^1 |f(t)-P(t)| \, \mathrm{d}t \leq \varepsilon$. Now by \eqref{eq:Qweak}, we may choose $n_0$ sufficiently large that for all $n \geq n_0$ we have
\[
\left| \int_0^1 f(Q_n(t)) \, \mathrm{d}t - \int_0^1 f(Q(t)) \, \mathrm{d}t \right| \leq \varepsilon.
\]
Then by the triangle inequality,
\begin{align} \label{eq:apple0}
\left| \int_0^1 P(Q_n(t)) - P(Q(t)) \, \mathrm{d}t \right| &\leq \left| \int_0^1 P(Q_n(t)) - f(Q_n(t)) \, \mathrm{d}t \right| + \left| \int_0^1 f(Q_n(t)) - f(Q(t)) \, \mathrm{d}t \right| \nonumber \\
& \qquad + \left| \int_0^1 P(Q(t)) - f(Q(t)) \, \mathrm{d}t \right|.
\end{align}
Now, by construction, for all $n \geq n_0$ we have 
\[
\left| \int_0^1 f(Q_n(t)) \, \mathrm{d}t - \int_0^1 f(Q(t)) \, \mathrm{d}t \right| \leq \varepsilon.
\]
Moreover, if $Q(t)$ is the quantile function of a measure $\pi$ with density bounded above by $K$, we have
\begin{align} \label{eq:apple}
\left| \int_0^1 P(Q(t)) - f(Q(t)) \, \mathrm{d}t \right| \leq \int_0^1 |P(x) - f(x)| \, \pi(\mathrm{d}x) \leq K \int_0^1 |P(x)-f(x)| \, \mathrm{d}x \leq K\varepsilon. 
\end{align}
The same bound as \eqref{eq:apple} holds with $Q_n(t)$ in place of $Q(t)$, and thus using \eqref{eq:apple0}, for all $n \geq n_0$ we have
\begin{align*}
\left| \int_0^1 P(Q_n(t)) - P(Q(t)) \, \mathrm{d}t \right|  \leq (2K+1) \varepsilon.
\end{align*}
Since $\varepsilon$ is arbitrary, \eqref{eq:latterconv} follows.

We turn to proving \eqref{eq:formerconv}. Let $Q(t)$ be the quantile function of a probability measure $\pi$ that is supported on $[0,1]$ and has a bounded density with respect to Lebesgue measure. Then
$$E(P,Q):= \int_0^1 P(Q(t)) \, \mathrm{d}t = \int_0^1 P(x) \, \pi(\mathrm{d}x) = \int_0^1 P(x)\rho(x) \, \mathrm{d}x$$
for some bounded function $\rho$. If $P_n \to P$ in $L^1$ then $P_n \rho \to P \rho$ in $L^1$, thus $E(P_n,Q) \to E(P,Q)$, completing the proof of \eqref{eq:formerconv}.
\end{proof}

In particular, we have the following corollary.
\begin{cor} \label{cor:trace}
The trace functional is continuous in the Wasserstein metric. That is, if $\mathbf{P}^n \to \mathbf{P}$ in Wasserstein distance, then
\begin{align*}
\lim_{n \to \infty} \mathrm{tr}(\mathbf{P}^n) = \mathrm{tr}(\mathbf{P}).
\end{align*}
\end{cor}
\begin{proof}
Recall that $\mathrm{tr}(\mathbf{P}) = \mathcal{E}(\mathbf{P},\mathbf{t})$. 
Set $\mathbf{Q} = \mathbf{t}$ in \eqref{eq:formerconv}.
\end{proof}

We will require one further lemma on specific limits of $\mathcal{E}$.

\begin{lem} \label{lem:Elimits}
Let $\mathbf{Q}$ be a triple of bounded quantile functions and let $\btQ$ be a triple of quantile functions taking values in the half-open interval $[0,1)$. Then
\begin{align}
\label{eqn:scalelimit}
   & \lim_{\varepsilon \downarrow 0} \mathcal{E}(\mathbf{Q}, \btQ + \varepsilon \mathbf{t}) = \mathcal{E}(\mathbf{Q}, \btQ).
\end{align}
\end{lem}

\begin{proof}

The assumption that the quantile functions $\btQ = (\tQ_1, \tQ_2, \tQ_3)$ take values in $[0,1)$ is required to ensure that $\mathcal{E}(\mathbf{Q}, \btQ + \varepsilon \mathbf{t})$ is well defined for sufficiently small $\varepsilon > 0$. Then it is enough to show that $E(Q_i, \tQ_i + \varepsilon t) \to E(Q_i, \tQ_i)$ for $E$ defined as in \eqref{eqn:E-single}. Since $Q_i$ is right-continuous, as $\varepsilon \downarrow 0$, the function $Q_i(\tQ_i(t) + \varepsilon t)$ converges pointwise to $Q_i(\tQ_i(t))$ for all $t \in [0,1]$. 
Since $Q_i$ is bounded, the functions $Q_i(\tQ_i(t) + \varepsilon t)$ are uniformly bounded for all $\varepsilon$, and the convergence $E(Q_i, \tQ_i + \varepsilon t) \to E(Q_i, \tQ_i)$ now follows from the bounded convergence theorem.
\end{proof}

\subsection{Approximation by Horn inequalities}

Recall that $\mathscr{H}[0,1]$ is the set of triples of probability measures supported on $[0,1]$ (equivalently, quantile functions taking values in $[0,1]$) that can be obtained as Wasserstein limits of empirical spectra of Hermitian triples $A+B=C$. The following theorem says not only that triples of the form $\mathbf{Q}_{I,J,K,n}$ associated with $(I,J,K) \in T_r^n$ are dense in $\mathscr{H}[0,1]$ as $n$ and $r$ range over all values, but also that any given element $\mathbf{Q} \in \mathscr{H}[0,1]$ may be approximated by a sequence $(I_n,J_n,K_n) \in T_{r_n}^n$ with $r_n/n$ converging to a prescribed ratio.  This is our first main result.

\densethm*

In fact, we will deduce Theorem \ref{thm:dense} from the following more quantitative statement.

\begin{thm} \label{thm:new23}
Let  $\mathbf{Q} \in \mathscr{H}[0,1]$, and let $\delta > 0$. Then for any pair $(n,r)$ with $36/\delta \le r \le n-1$, there exists an $n$-integral and $r$-atomic triple $\mathbf{Q}' \in \mathscr{H}[0,1-\eta_{\mathbf{Q}}+\delta]$ that satisfies
\begin{align*}
\int_0^1 |Q_i'(t) - Q_i(t)| \, \mathrm{d}t < \delta \quad \text{for each $i=1,2,3$.}
\end{align*}
If it is also the case that $\frac{r}{n} \leq \eta_{\mathbf{Q}}-\delta$, then $\mathbf{Q}' = \mathbf{Q}_{I,J,K,n}$ for some $(I,J,K) \in T_r^n$.
\end{thm}

In order to prove Theorem \ref{thm:new23}, we first show two lemmas. In Lemma \ref{lem:S}, we construct a particular $r$-atomic triple that uniformly satisfies the Horn inequalities indexed by $T_p^r$ for $1 \leq p \leq r-1$. We use this construction in the proof of the following lemma, Lemma \ref{lem:semilatapp}, which shows that any $r$-atomic solution of the Horn inequalities may be well-approximated by $n$-integral and $r$-atomic solutions for large values of $n$.

\begin{lem} \label{lem:S}
Define an $r$-atomic triple $\mathbf{S} = (S_1(t),S_2(t),S_3(t))$ by setting, for $i=1,2,3$,
\begin{align} \label{eq:crta}
S_i(t) := s + \mathbb{1}_{i=3} \frac{r+1}{2}\qquad \text{whenever $t \in \left[ \frac{s-1}{r}, \frac{s}{r} \right)$ for some $s \in \{1,\ldots,r\}$}.
\end{align}
Then $\mathrm{tr}(\mathbf{S}) = 0$, and 
\begin{align} \label{eq:set1}
\mathcal{E} \left(\mathbf{S} ,  \mathbf{Q}_{F,G,H,r} + \frac{p}{r}\mathbf{t}\right) \geq r/2  \quad \text{for every $1 \leq p \leq r-1 $ and every $(F,G,H) \in T_p^r$. }
\end{align}
\end{lem}
By Proposition \ref{prop:forward}, Lemma \ref{lem:S} implies that $\mathbf{S}$ lies in $\mathscr{H}_r(\mathbb{R})$. 

\begin{proof}
The proof that $\mathrm{tr}(\mathbf{S}) = 0$ is a straightforward calculation. 

We turn to proving \eqref{eq:set1}. 
Using Lemma \ref{lem:fincom} to obtain the second equality below, and then \eqref{eq:crta} to obtain the third, we have 
\begin{align*}
\mathcal{E} \left(\mathbf{S} ,  \mathbf{Q}_{F,G,H,r} + \frac{p}{r}\mathbf{t}\right) &= - \int_0^1 S_1 \big( Q_{F,r}(t) + \frac{p}{r}t \big) \, \mathrm{d}t  - \int_0^1 S_2 \big( Q_{G,r}(t)+ \frac{p}{r}t \big) \, \mathrm{d}t\\
& \qquad + \int_0^1 S_3 \big( Q_{H,r}(t)+ \frac{p}{r}t \big) \, \mathrm{d}t \\
&=\frac{r}{p} \left( - \sum_{f \in F} \int_{(f-1)/r}^{f/r} S_1(t) \, \mathrm{d}t - \sum_{g \in G} \int_{(g-1)/r}^{g/r} S_2(t) \, \mathrm{d}t + \sum_{h \in H} \int_{(h-1)/r}^{h/r} S_3(t) \, \mathrm{d}t \right)\\
&= \frac{r}{p} \left( - \sum_{f \in F} f - \sum_{ g \in G} g + \sum_{ h \in H} \left( h + \frac{r+1}{2} \right) \right). 
\end{align*}
For $(F,G,H) \in T_p^r \subset U_p^r$, the definition \eqref{eqn:Unr-def} implies that
$$- \sum_{f \in F} f - \sum_{ g \in G} g + \sum_{ h \in H} h = - \frac{p(p+1)}{2},$$
and so, together with the fact that $H$ has cardinality $p$, we have 
\begin{align} \label{eq:newlower2}
\mathcal{E} \left(\mathbf{S} ,  \mathbf{Q}_{F,G,H,r} + \frac{p}{r}\mathbf{t}\right) = \frac{r}{p} \left( - \frac{p(p+1)}{2} + p \frac{r+1}{2} \right) = \frac{r(r-p)}{2} \geq r/2,
\end{align} 
where in the final inequality above we used $p \leq r-1$. That completes the proof.
\end{proof}

The next lemma states that $r$-atomic solutions of the Horn inequalities can be well approximated by $n$-integral and $r$-atomic solutions.

\begin{lem} \label{lem:semilatapp}

Let $\mathbf{Q}$ be an $r$-atomic triple in $\mathscr{H}_r(\mathbb{R})$. Then for every $\delta > 0$ and every $n \geq 18/\delta$, there exists an $n$-integral and $r$-atomic triple $\mathbf{Q}' \in \mathscr{H}_r^n(\mathbb{R})$ approximating $\mathbf{Q}$ in the sense that 
\begin{align} \label{eq:delta1}
|Q'_i(t) -Q_i(t)| \leq \delta \qquad \text{for all $t \in [0,1], \ i=1,2,3$}.
\end{align}
If $\mathbf{Q}$ is nonnegative, i.e.\ if $Q_i(t) \geq 0$ for all $t \in [0,1]$ and $i =1,2,3$, then one can take $\mathbf{Q}'$ to be nonnegative as well.
\end{lem}
\begin{proof}
We would like to approximate $\mathbf{Q}$ by an $n$-integral and $r$-atomic triple in $\mathscr{H}_r^n(\mathbb{R})$, but first we obtain some leeway by adding a small perturbation to $\mathbf{Q}$ so as to ensure that the relevant Horn inequalities hold in a way uniformly bounded away from zero. 
To this end, we use the triple $\mathbf{S}$ constructed in Lemma \ref{lem:S} to define a new $r$-atomic triple $\mathbf{Q}^{(\varepsilon)}$ by setting 
\begin{align} \label{eq:S0def}
\mathbf{Q}^{(\varepsilon)} := \mathbf{Q}+\varepsilon \mathbf{S}.
\end{align}
Noting that $0 \leq S_i(t) \leq (3r+1)/2$ for all $t \in [0,1]$ and $i=1,2,3$, we have
\begin{align} \label{eq:tride}
|Q_i^{(\varepsilon)}(t) - Q_i(t)| \leq (3r+1)\varepsilon/2 \leq 2r\varepsilon \qquad \text{for all $t \in [0,1], \ i=1,2,3$}.
\end{align}
By the linearity of $\mathcal{E}( \cdot , \btQ)$ with respect to addition of quantile functions, for $(F,G,H) \in T_p^r$ we have
\begin{align} \label{eq:newlower1}
\mathcal{E}\Big( \mathbf{Q}^{(\varepsilon)} , \mathbf{Q}_{F,G,H,r} + \frac{p}{r}\mathbf{t} \Big) &= \mathcal{E}\Big(\mathbf{Q}, \mathbf{Q}_{F,G,H,r} + \frac{p}{r}\mathbf{t} \Big) + \varepsilon \mathcal{E} \left(\mathbf{S} ,  \mathbf{Q}_{F,G,H,r} + \frac{p}{r}\mathbf{t}\right) \geq \varepsilon r /2,
\end{align}
where to obtain the final inequality above we used the fact that $\mathbf{Q}$ satisfies the Horn inequalities together with \eqref{eq:set1}. Thus $\mathbf{Q}^{(\varepsilon)} = (Q_1^{(\varepsilon)}(t),Q_2^{(\varepsilon)}(t),Q_3^{(\varepsilon)}(t))$ satisfies the relevant Horn inequalities in a way uniformly bounded away from zero. 

For all sufficiently large $n$, we now construct a triple $\mathbf{Q}' = \mathbf{Q}^{(\varepsilon,n)} \in \mathscr{H}_r^n(\mathbb{R})$ that satisfies the conditions of the theorem. Given any integer $n \geq 1$, we begin by defining an auxillary triple $\mathbf{P}^{(\varepsilon,n)}$ by setting 
\begin{align} \label{eq:P0def}
P_i^{(\varepsilon,n)}(t) := \text{Largest multiple of $1/n$ less than $Q_i^{(\varepsilon)}(t)$}.
\end{align}
Then $\mathbf{P}^{(\varepsilon,n)} := (P_i^{(\varepsilon,n)}(t) , P_2^{(\varepsilon,n)}(t)  , P_3^{(\varepsilon,n)}(t)  )$ is a triple of $n$-integral and $r$-atomic quantile functions satisfying
\begin{align} \label{eq:closeness}
P_i^{(\varepsilon,n)}(t) \leq Q_i^{(\varepsilon)}(t) \leq P_i^{(\varepsilon,n)}(t) + 1/n
\end{align}
 for each $t \in [0,1]$ and $i=1,2,3$. However, in general $\mathrm{tr}(\mathbf{P}^{(\varepsilon,n)} ) \neq 0$. The remainder of the proof is dedicated to remedying this issue. 

 By \eqref{eq:closeness}, the definition \eqref{eq:trace} of $\mathrm{tr}( \cdot)$, and the fact that $\mathrm{tr}(\mathbf{Q}^{(\varepsilon)}) = \mathrm{tr}(\mathbf{Q}) + \varepsilon \mathrm{tr}(\mathbf{S}) = 0$, we have
\begin{align} \label{eq:sandy}
- 1/n \leq \mathrm{tr}(\mathbf{P}^{(\varepsilon,n)}) \leq 2/n.
\end{align}
Also, since each $P_i^{(\varepsilon,n)}(t)$ is constant on intervals of length $1/r$ and takes values in integer multiples of $1/n$, it follows that $\mathrm{tr}(\mathbf{P}^{(\varepsilon,n)}) = \tau_{\varepsilon,n}/nr$ for some integer $\tau_{\varepsilon,n}$. By \eqref{eq:sandy}, we have $- r \leq \tau_{\varepsilon,n} \leq 2r$. 
Now choose three integers $0 \leq a_1, a_2,a_3 \leq r$ such that $a_1 + a_2 - a_3 = \tau_{\varepsilon,n}$, and define
\begin{align} \label{eq:Q0def}
\mathbf{Q}^{(\varepsilon,n)}_i(t) :=\mathbf{P}^{(\varepsilon,n)}_i(t) + \frac{1}{n} \mathbb{1} \left\{ t \geq \frac{r-a_i}{r} \right\}.
\end{align}
Then $\mathrm{tr}(\mathbf{Q}^{(\varepsilon,n)}) = \mathrm{tr}(\mathbf{P}^{(\varepsilon,n)}) -
 \frac{1}{nr}( a_1 + a_2 - a_3 ) = 0$ by construction. Note also that by \eqref{eq:closeness} and the fact that $\mathbf{Q}^{(\varepsilon,n)}_i(t)$ equals either $\mathbf{P}^{(\varepsilon,n)}_i(t)$ or $\mathbf{P}^{(\varepsilon,n)}_i(t) + 1/n$, it follows that
\begin{align} \label{eq:Qsw}
|Q_i^{(\varepsilon,n)}(t) - Q_i^{(\varepsilon)}(t)| \leq 1/n.
\end{align}
Combining \eqref{eq:Qsw} with \eqref{eq:tride}, this implies that 
\begin{align} \label{eq:tride2}
|Q_i^{(\varepsilon, n)}(t) - Q_i(t)| \leq 2r\varepsilon+1/n \qquad \text{for all $t \in [0,1], \ i=1,2,3$}.
\end{align}

We have thus far constructed 
an $n$-integral and $r$-atomic triple $\mathbf{Q}^{(\varepsilon,n)}$ satisfying $\mathrm{tr}(\mathbf{Q}^{(\varepsilon,n)}) = 0$ and approximating $\mathbf{Q}$ in the sense that \eqref{eq:tride2} holds. 
To establish that $\mathbf{Q}^{(\varepsilon,n)}$ is an element of $\mathscr{H}_r^n(\mathbb{R})$, we need to verify that it respects all of the Horn inequalities.

To this end, note from the definition \eqref{eq:energy} of the composition functional $\mathcal{E}(\mathbf{Q},\tilde{\mathbf{Q}})$ that if $\mathbf{P}$ and $\mathbf{Q}$ are triples satisfying $|Q_i(t)-P_i(t)| \leq \rho$ for all $t \in [0,1],$ $i=1,2,3$, then for any other triple $\btQ$ taking values in $[0,1]$ we have 
\begin{align} \label{eq:energyclose}
|\mathcal{E}(\mathbf{Q},\btQ) - \mathcal{E}(\mathbf{P}, \btQ )| \leq 3 \rho.
\end{align}
Suppose $n \geq 6/\varepsilon r$. Then using \eqref{eq:Qsw} in conjunction with \eqref{eq:energyclose} to obtain the first inequality below, 
\eqref{eq:newlower1} to obtain the second, and $n \geq 6/\varepsilon r$ to obtain the third, for any $(F,G,H) \in T_p^r$ we have
\begin{align*}
\mathcal{E}\Big( \mathbf{Q}^{(\varepsilon,n)} , \mathbf{Q}_{F,G,H,r} + \frac{p}{r} \mathbf{t} \Big) \geq \mathcal{E}\Big( \mathbf{Q}^{(\varepsilon)} , \mathbf{Q}_{F,G,H,r} + \frac{p}{r} \mathbf{t} \Big) - 3/n \geq \varepsilon r/2 - 3/n \geq 0.
\end{align*}
Thus we have established that $ \mathbf{Q}^{(\varepsilon,n)}$ satisfies every Horn inequality indexed by a triple $(F,G,H)$ in $T_p^r$ for some $1 \leq p \leq r-1$. By Proposition \ref{prop:forward}, since $\mathbf{Q}^{(\varepsilon,n)}$ is $r$-atomic, $ \mathbf{Q}^{(\varepsilon,n)}$ respects every Horn inequality indexed by any $(I,J,K)$ in any $T_{r'}^{n'}$. That is, $\mathbf{Q}^{(\varepsilon,n)}$ is an element of $\mathscr{H}_r^n(\mathbb{R})$. 

Using $n \geq 6/\varepsilon r$ in \eqref{eq:tride2} we have
\begin{align} \label{eq:tride3}
|Q_i^{(\varepsilon, n)}(t) - Q_i(t)| \leq 3r\varepsilon \qquad \text{for all $t \in [0,1], \ i=1,2,3$}.
\end{align}
In summary, we have shown that for any $\varepsilon > 0$ and any $n \geq 6/\varepsilon r$ we may construct an element of $\mathscr{H}_r^n(\mathbb{R})$ satisfying \eqref{eq:tride3}. Setting $\delta = 3r\varepsilon$ we can obtain \eqref{eq:delta1} as written.

Finally, we note that by the definitions \eqref{eq:S0def}, \eqref{eq:P0def} and \eqref{eq:Q0def}, nonnegativity of $\mathbf{Q}$ implies that each of the constructed triples $\mathbf{Q}^{(\varepsilon)}$, $\mathbf{P}^{(\varepsilon,n)}$ and $\mathbf{Q}' = \mathbf{Q}^{(\varepsilon,n)}$ are also nonnegative. That completes the proof.
\end{proof}


With the preceding groundwork, Theorems \ref{thm:new23} and \ref{thm:dense} are now quick to prove.

\begin{proof}[Proof of Theorem \ref{thm:new23}]
Let $\mathbf{Q}$ be any triple of quantile functions in $\mathscr{H}[0,1]$.

Let $n,r$ be any integers with $36/\delta \le r \le n-1$. Let $\mathbf{Q}^r$ be the $r$-atomic average of $\mathbf{Q}$, so that by Proposition \ref{prop:selfavg}, $\mathbf{Q}^r$ belongs to $\mathscr{H}_r(\mathbb{R})$, and by Lemma \ref{lem:avg-L1}, \mbox{$||Q^r_i - Q_i||_1 \leq 1/r$}. 

Let $\sigma = \delta/2$. Note that since $n > r \geq 36/\delta = 18/\sigma$, by Lemma \ref{lem:semilatapp} there exists a triple in $\mathscr{H}_r^n(\mathbb{R})$ satisfying $|Q'_i(t)-Q_i^r(t)| \leq \sigma$ for all $t \in [0,1]$ and $i=1,2,3$, which in particular implies that $||Q_i' - Q_i^r||_1 \leq \sigma$. 
It now follows from the triangle inequality that
\begin{align*}
||Q_i - Q_i'||_1 & \leq ||Q_i-Q_i^r||_1+||Q_i^r-Q_i'||_1\\
&\leq  ||Q_i-Q_i^r||_1+||Q_i^r-Q_i'||_\infty \leq 1/r + \sigma \leq \delta / 36 + \delta/2 < \delta,
\end{align*}
which proves the desired inequality.

Note that the quantile functions in the triple $\mathbf{Q}^r$ also take values in $[0,1-\eta_{\mathbf{Q}}]$. Since, by Lemma \ref{lem:semilatapp}, $\mathbf{Q}^r$ nonnegative implies $\mathbf{Q}'$ nonnegative, and $|Q'_i(t)-Q^r_i(t)| \leq \sigma$ for each $i$ and $t$, it follows that $\mathbf{Q}'$ takes values in $[0,1-\eta_{\mathbf{Q}}+\sigma] \subset [0,1-\eta_{\mathbf{Q}}+\delta]$. 

If $\frac{r}{n} \leq \eta_{\mathbf{Q}}-\delta$, then 
$[0,1-\eta_{\mathbf{Q}}+\delta] \subseteq [0,1-r/n]$. In this case $\mathbf{Q}'$ is an element of $\mathscr{H}_r^n[0,1-r/n]$. By Lemma \ref{lem:Tmember}, $\mathbf{Q}' = \mathbf{Q}_{I,J,K,n}$ for some $(I,J,K) \in T_r^n$, completing the proof.
\end{proof}

\begin{proof}[Proof of Theorem \ref{thm:dense}]
We first prove the statement for $\mathbf{Q} \in \mathscr{H}[0,1)$ and $q \in [0,\eta_{\mathbf{Q}})$. We then extend it to the case $q = \eta_{\mathbf{Q}}$, and finally to the case $\mathbf{Q} \in \mathscr{H}[0,1] - \mathscr{H}[0,1)$.

Choose $\mathbf{Q} \in \mathscr{H}[0,1)$, $q \in [0,\eta_{\mathbf{Q}})$, and a sequence $r_n$ such that $r_n/n \to q$. For $\delta$ sufficiently small and $n$ sufficiently large depending on $\delta$, we have $r_n \ge 36/\delta$ and $r_n/n \le \eta_{\mathbf{Q}}-\delta$. Then by Theorem \ref{thm:new23}, there exists $\mathbf{Q}' = \mathbf{Q}_{I,J,K,n}$ for some $(I,J,K) \in T_{r_n}^n$ such that $\| Q_i' - Q_i \|_1 \leq \delta$; this completes the proof for $\mathbf{Q} \in \mathscr{H}[0,1)$ and $q \in [0,\eta_{\mathbf{Q}})$.

For the case $\mathbf{Q} \in \mathscr{H}[0,1)$, $q = \eta_{\mathbf{Q}}$, we make a diagonalisation argument. By applying the result for the limiting ratio $r_n/n \to q' = \eta_{\mathbf{Q}} - \varepsilon/2$ with $0 < \varepsilon < 2 \eta_{\mathbf{Q}}$, we see that for $n$ sufficiently large there exists $\mathbf{Q}' = \mathbf{Q}_{I,J,K,n}$ with $(I,J,K) \in T_{r_n}^n$, with $|q' - r_n/n| < \varepsilon/2$, and such that $||Q'_i - Q_i||_1 < \varepsilon$. We then also have $|\eta_{\mathbf{Q}} - r_n/n| < \varepsilon$. Considering a sequence of such triples $\mathbf{Q}'$ while taking $\varepsilon \to 0$ completes the proof for $\mathbf{Q} \in \mathscr{H}[0,1)$ and $q = \eta_{\mathbf{Q}}$.

It remains only to handle the case $\mathbf{Q} \in \mathscr{H}[0,1] - \mathscr{H}[0,1)$.  For such $\mathbf{Q}$ we necessarily have $\eta_{\mathbf{Q}} = 0$, so we need only consider $q = 0$. By the dilation invariance described in Section \ref{sec:dil-trans}, for $0 < \varepsilon < 1$ we have $(1-\varepsilon) \mathbf{Q} \in \mathscr{H}[0,1-\varepsilon]$. The proof of the final case now follows via a further diagonalisation argument, using the first case above to approximate $(1-\varepsilon)\mathbf{Q}$ while sending $\varepsilon \to 0$.
\end{proof}

\subsection{The self-characterisation property}

We are now equipped to prove our second main result, on self-characterisation, which we restate below. Recall that $\eta_\mathbf{Q} := 1 - \max_{i=1,2,3} \sup_{t \in [0,1]} Q_i(t)$.

\selfcharthm*

\begin{proof}
We first show that the second statement giving a criterion for membership in $\mathscr{H}[0,1]$ follows from the first statement giving a criterion for membership in $\mathscr{H}(\R)$. Concretely, we prove that if $\mathcal{E}(\mathbf{Q}, \btQ + \mu \mathbf{t}) \ge 0$ for all $\mathbf{Q} \in \mathscr{H}[0,1]$, $\btQ \in \mathscr{H}[0,1)$, and $0 < \mu \le \eta_\btQ$, then in fact $\mathcal{E}(\mathbf{Q}, \btQ + \mu \mathbf{t}) \ge 0$ for all $\mathbf{Q}, \btQ \in \mathscr{H}[0,1]$ and $0 \le \mu \le \eta_\btQ$.

Note that if $\mathbf{Q}$ is a triple of bounded quantile functions and $\mathcal{E}(\mathbf{Q}, \btQ + \mu \mathbf{t}) \ge 0$ for $0 < \mu \le \eta_\btQ$, then sending $\mu \to 0$ we also obtain $\mathcal{E}(\mathbf{Q}, \btQ) \ge 0$ by Lemma \ref{lem:Elimits}. Then, observing that the triples $\btQ \in \mathscr{H}[0,1] - \mathscr{H}[0,1)$ are precisely those for which $\eta_{\btQ} = 0$, it is enough to show that
\begin{align} \begin{split} \label{eqn:H-crit-red}
\mathcal{E}(\mathbf{Q},\btQ) \geq 0 \quad &\forall \, \mathbf{Q} \in \mathscr{H}[0,1], \ \btQ \in \mathscr{H}[0,1) \\ \implies \quad \mathcal{E}(\mathbf{Q},\btQ) \geq 0 \quad &\forall \, \mathbf{Q}, \btQ \in \mathscr{H}[0,1].
\end{split} \end{align}

Let $\mathbf{Q}, \btQ \in \mathscr{H}[0,1]$, and let $\boldsymbol{\pi} = (\pi_1, \pi_2,\pi_3)$ be the triple of probability measures represented by $\frac{1}{2}\mathbf{Q}$. By dilation invariance --- the fact that $\mathscr{H}(\mathbb{R})$ is closed under scaling the support of measures as discussed in Section \ref{sec:dil-trans} --- we have $\boldsymbol{\pi} \in \mathscr{H}[0,1/2]$. Let $\boldsymbol{\pi}' = (\delta_{1/2}, \, \delta_{1/2}, \, \delta_1) \in \mathscr{H}[0,1]$ and write $\mathbf{Q}' = (Q_1', Q_2', Q_3')$ for the quantile functions of $\boldsymbol{\pi}'$, that is,
\begin{align*}
Q_1'(t) &\equiv 1/2, \\
Q_2'(t) &\equiv 1/2, \\
Q_3'(t) &\equiv 1, \qquad \quad t \in [0,1].
\end{align*}
By the vertical convexity property in Proposition \ref{prop:convexity}, the average $\frac{1}{2}(\boldsymbol{\pi} + \boldsymbol{\pi}')$ belongs to $\mathscr{H}[0,1]$, with quantile functions $\mathbf{P} = (P_1, P_2, P_3)$ given by
\[
P_i(t) = \begin{cases}
\frac{1}{2} Q_i(2t), & t \in \big[0, \frac{1}{2} \big), \\
Q'_i(2t - 1), & t \in \big[ \frac{1}{2}, 1 \big].
\end{cases}
\]
Again by dilation invariance, $\frac{1}{2} \btQ \in \mathscr{H}[0,1)$.  From the definition \eqref{eq:energy} of $\mathcal{E}$, we find
\[
\mathcal{E}(\mathbf{Q}, \btQ) = 2 \mathcal{E}\Big(\mathbf{P}, \frac{1}{2} \btQ \Big)
\]
where $\mathbf{P} \in \mathscr{H}[0,1]$ and $\frac{1}{2} \btQ \in \mathscr{H}[0,1)$.  This proves \eqref{eqn:H-crit-red}, showing that the second statement in the theorem follows from the first statement.

\vspace{3mm}
It remains to prove the criterion for membership in $\mathscr{H}(\R)$:
\[
\mathbf{Q} \in \mathscr{H}(\R) \iff \mathcal{E}(\mathbf{Q}, \btQ + \mu \mathbf{t}) \ge 0 \quad \forall \, \btQ \in \mathscr{H}[0,1), \ 0 < \mu \le \eta_\btQ.
\]

We first suppose that $\mathbf{Q} \in \mathscr{H}(\R)$ and show that the inequalities hold. By the definition of $\mathscr{H}(\R)$, we can choose a sequence $(\mathbf{Q}^m)_{m \ge 1}$ of triples of quantile functions of empirical spectra of $m$-by-$m$ Hermitian matrices $A_m + B_m = C_m$, such that $\mathbf{Q}^m \to \mathbf{Q}$.  Additionally, for any $\btQ \in \mathscr{H}[0,1)$ and $0 < \mu \le \eta_\btQ$, by Theorem \ref{thm:dense} we can choose a sequence $(I_n,J_n,K_n)_{n \geq 1} \in T_{p_n}^{r_n}$ such that $\mathbf{Q}_{I_n,J_n,K_n,r_n} \to \btQ$ and $p_n/r_n \to \mu$. Note that this implies 
\begin{align} \label{eq:crab}
\mathbf{Q}_{I_n,J_n,K_n,r_n}+\frac{p_n}{r_n} \mathbf{t} \ \to \ \btQ + \mu \mathbf{t}.
\end{align}

Now by Propositions \ref{prop:hornexp} and \ref{prop:forward}, we have 
\begin{align*}
\mathcal{E}\Big( \mathbf{Q}^m,\mathbf{Q}_{I_n,J_n,K_n,r_n} + \frac{p_n}{r_n} \mathbf{t} \Big) \geq 0
\end{align*}
for all $m$ and $n$. Note that $\mathbf{Q}_{I_n,J_n,K_n,r_n} + \frac{p_n}{r_n} \mathbf{t}$ represents a triple of probability measures with densities bounded above by $r_n/p_n$, and this ratio converges to $1/\mu$. In particular, there is some $n_0$ such that for all $n \geq n_0$, the density of each measure corresponding to a quantile function in the triple $\mathbf{Q}_{I_n,J_n,K_n,r_n} + \frac{p_n}{r_n} \mathbf{t}$ is bounded above by $2/\mu$. Thus by \eqref{eq:crab} we may apply equation \eqref{eq:latterconv} of Lemma \ref{lem:continuity} to conclude that for each fixed $m$ we have 
\begin{align*}
\lim_{n \to \infty} \mathcal{E}\Big( \mathbf{Q}^m,\mathbf{Q}_{I_n,J_n,K_n,r_n} + \frac{p_n}{r_n} \mathbf{t} \Big) = \mathcal{E}( \mathbf{Q}^m, \btQ + \mu \mathbf{t} ) \ge 0, 
\end{align*}
where the inequality holds because every term in the sequence indexed by $n$ is nonnegative.

We next note that $\btQ + \mu \mathbf{t}$ represents a triple of measures with densities bounded above by $1/\mu$. Thus we may apply equation \eqref{eq:formerconv} of Lemma \ref{lem:continuity} to conclude that
\begin{align*}
\lim_{m \to \infty} 
 \mathcal{E}( \mathbf{Q}^m, \btQ + \mu \mathbf{t} )  =
 \mathcal{E}( \mathbf{Q}, \btQ + \mu \mathbf{t} ) \ge 0,
\end{align*} 
where again the inequality holds because each term in the sequence is nonnegative. This completes the proof of the first direction of implication.

\vspace{3mm}
Now we prove the other direction. Namely, we show that if $\mathbf{Q}$ is a triple of integrable quantile functions satisfying $\mathrm{tr}(\mathbf{Q}) = 0$ and with the property that for all $\btQ \in \mathscr{H}[0,1)$ and $0 < \mu \le \eta_\btQ$ we have $\mathcal{E} ( \mathbf{Q} , \btQ + \mu \mathbf{t} ) \geq 0$, then $\mathbf{Q} \in \mathscr{H}(\R)$. In other words, we show that $\mathbf{Q} = (Q_1(t),Q_2(t),Q_3(t))$ arises as a limit of triples of quantile functions of empirical spectra of matrices $A_n+B_n=C_n$.

To this end, for each $n \geq 1$, $i= 1, 2, 3$, and $t$ in an interval $\left[ \frac{j-1}{n}, \frac{j}{n} \right)$, we let $Q_i^n(t)$ be the average of $Q_i(t)$ on this interval:
\begin{align*}
Q_i^n(t) := n \int_{(j-1)/n}^{j/n} Q_i(s) \, \mathrm{d}s \qquad \text{for } t \in \left[ \frac{j-1}{n}, \frac{j}{n} \right).
\end{align*}
Then $Q_i^n \to Q_i$ in $L^1$ by Lemma \ref{lem:avg-L1}.  Moreover by Proposition \ref{prop:selfavg}, for each $n \geq 1$, $\mathbf{Q}^n$ are the spectral quantile functions of a triple of Hermitian matrices $A_n + B_n = C_n$, and thus $\mathbf{Q} \in \mathscr{H}(\R)$.
\end{proof}

As a consequence of Theorem \ref{thm:SC}, we find that if a triple of compactly supported probability measures satisfies the trace condition and all finite-$n$ Horn inequalities, then in fact it satisfies all inequalities corresponding to points in the asymptotic Horn system.

\begin{cor} \label{cor:cpt-supp}
    Let $\mathbf{Q}$ be a triple of quantile functions of compactly supported measures with $\mathrm{tr}(\mathbf{Q}) = 0$.  If $\mathbf{Q}$ satisfies the Horn inequalities, that is, if $\mathcal{E}(\mathbf{Q}, \mathbf{Q}_{I,J,K,n} + \frac{r}{n} \mathbf{t}) \ge 0$
    for all $(I,J,K)$ in all $T^n_r$, then in fact $\mathcal{E}(\mathbf{Q}, \btQ + \mu \mathbf{t}) \ge 0$ for all $\btQ \in \mathscr{H}[0,1]$ and $0 \le \mu \le \eta_\btQ$.
\end{cor}

\begin{proof}
    If $\mathcal{E}(\mathbf{Q}, \mathbf{Q}_{I,J,K,n} + \frac{r}{n} \mathbf{t}) \ge 0$
    for all $(I,J,K)$ in all $T^n_r$, then $\mathbf{Q} \in \mathscr{H}(\R)$ by Proposition \ref{prop:H-ineqs}.  The assumption that $\mathbf{Q}$ corresponds to a triple of compactly supported measures is equivalent to asssuming that each $Q_i$ is bounded.  Let
    \begin{align*}
    s &= \sup_{i=1,2,3;\ t \in [0,1]} \big |Q_i(t) \big|.
    \end{align*}
    We may assume that $s \ne 0$, since otherwise we would have $\mathbf{Q} = (0,0,0) \in \mathscr{H}[0,1]$ and we would be done.  Since $\mathbf{Q} \in \mathscr{H}(\R)$, the triple $\mathbf{Q}' = \frac{1}{3s}(Q_1(t) + s, \, Q_2(t) + s, \, Q_3(t) + 2s)$ also belongs to $\mathscr{H}(\R)$ by dilation and translation invariance as discussed in Section \ref{sec:dil-trans}, and each $Q'_i$ takes values in $[0,1]$ by construction, so in fact  $\mathbf{Q}' \in \mathscr{H}[0,1]$. Thus by Theorem \ref{thm:SC}, $\mathcal{E}(\mathbf{Q}', \btQ + \mu \mathbf{t}) \ge 0$ for all $\btQ \in \mathscr{H}[0,1]$ and $0 \le \mu \le \eta_\btQ$.  Since $\mathcal{E}(\mathbf{Q}, \btQ + \mu \mathbf{t}) = 3s \mathcal{E}(\mathbf{Q}', \btQ + \mu \mathbf{t})$, this proves the claim.
\end{proof}

\begin{remark}
    There are some similarities between the self-characterisation property and the notion of self-duality for cones in a finite-dimensional Euclidean space.  We point this out in order to caution the reader, because the analogy only goes so far and may be misleading.  These similarities become more evident if we keep track of the value of $\mu$ as an additional coordinate, because we then can state the self-characterisation property in such a way that each linear inequality on the quantile functions corresponds to a distinct point in a convex body. Define
    \begin{align*}
    \mathscr{H}^\bullet[0,1] &:= \big\{  (\mathbf{Q},\mu) \ \big | \ \mathbf{Q} \in \mathscr{H}[0,1] , \, 0 \le \mu \le \eta_{\mathbf{Q}} \big\}, \\
    \mathcal{E}^\bullet \big((\mathbf{Q},\mu), (\btQ,\tilde{\mu}) \big) &:= \mathcal{E}(\mathbf{Q} + \mu \mathbf{t}, \btQ + \tilde{\mu} \mathbf{t}) - \frac{1}{2} \mu \tilde{\mu} = \mathcal{E}(\mathbf{Q}, \btQ + \tilde{\mu} \mathbf{t}).
    \end{align*}
    Then for $\mathbf{Q}$ a triple of quantile functions taking values in $[0,1]$ with $\tr(\mathbf{Q}) = 0$, and for $0 \le \mu \le \eta_{\mathbf{Q}}$, the self-characterisation property can be stated as
    \[
    (\mathbf{Q}, \mu) \in \mathscr{H}^\bullet[0,1]  \iff \mathcal{E}^\bullet \big((\mathbf{Q},\mu), (\btQ,\tilde{\mu}) \big)  \ge 0 \quad \forall \, (\btQ,\tilde{\mu}) \in \mathscr{H}^\bullet[0,1].
    \]
    The functional $\mathcal{E}^\bullet$ can be interpreted as a bilinear pairing via the two different linear structures described in Section \ref{sec:convexity}: the linearity in the first argument comes from addition of quantile functions of probability measures on $\R$, while the linearity in the second argument comes from addition of finite signed measures on $[0,1]$. In that light, the above statement superficially resembles the definition of a self-dual cone. However, the analogy quickly breaks down.  To begin with, although $\mathscr{H}^\bullet[0,1]$ is convex, it is not a cone, due to the requirements that each $Q_i$ take values in $[0,1]$ and that $0 \le \mu \le \eta_{\mathbf{Q}}$.  Moreover, even though $\mathcal{E}^\bullet$ can be regarded as a bilinear pairing, it is a discontinuous and degenerate pairing between two different vector spaces, neither of which is contained in the topological dual of the other.  To further complicate matters, the correspondence between quantile functions and probability measures is nonlinear, and $\mathcal{E}(\mathbf{Q}, \btQ)$ is only defined when each $\tQ_i$ takes values in $[0,1]$, so the map
    \[
    (\btQ,\tilde{\mu}) \mapsto \mathcal{E}^\bullet \big( \, \cdot \, , (\btQ,\tilde{\mu}) \big)
    \]
    is not a linear transformation from the relevant vector space to its dual. The self-characterisation property of the asymptotic Horn system therefore does not correspond to a straightforward notion of ``self-duality'' as in finite-dimensional Euclidean spaces.  Nevertheless, self-dual cones \textit{are} self-characterising in the more general sense that we define below in Section \ref{sec:SCunique}; see Example \ref{ex:cones}.
\end{remark}

\subsection{Redundancy in the infinite-dimensional Horn inequalities}
In this section we prove our final main result, Theorem \ref{thm:redundancy}, which identifies many proper subsets of Horn inequalities that are equivalent to the entire system in the sense that they have the same set of solutions:

\redundancy*

\begin{proof}
The ``only if'' direction is immediate from Theorem \ref{thm:SC}.

For the ``if'' direction, we show that if $\mathbf{Q}^n$ is the $n$-average of the triple $\mathbf{Q}$, then \eqref{eqn:thin-horn} implies that $\mathbf{Q}^n$ lies in $\mathscr{H}(\mathbb{R})$ for every $n$. Since $\mathscr{H}(\mathbb{R})$ is closed and $\mathbf{Q}^n$ converges to $\mathbf{Q}$ (Lemma \ref{lem:avg-L1}), we then can conclude that $\mathbf{Q}$ lies in $\mathscr{H}(\mathbb{R})$ as well. In this direction, we begin by showing that if the inequalities (\ref{eqn:thin-horn}) hold, then for $n$ sufficiently large,
\[
\mathcal{E}\Big(\mathbf{Q}^n, \mathbf{Q}_{I,J,K,n} + \frac{r}{n} \mathbf{t} \Big) \ge 0 \quad \text{ for all $(I,J,K) \in T^{n}_{r}$, $1 \le r \le n-1$.}
\]
Note that the assumptions imply that $n_k > r_k$ and that $r_k$ is unbounded.  By Theorem \ref{thm:new23}, for any $(I,J,K) \in T^{n}_{r}$ we can choose subsequences $(n_{k_j})_{j \ge 1}$, $(r_{k_j})_{j \ge 1}$ and a sequence $(F_j, G_j, K_j) \in T^{n_{k_j}}_{r_{k_j}}$ such that $r_{k_j} / n_{k_j} \to r/n$ and
\[
\mathbf{Q}_{F_j,G_j,K_j,n_{k_j}} + \frac{r_{k_j}}{n_{k_j}} \mathbf{t} \ \ \longrightarrow \ \ \mathbf{Q}_{I,J,K,n} + \frac{r}{n} \mathbf{t}.
\]
Since the quantile functions in the sequence of triples above correspond to measures with densities that are eventually bounded above by $2n/r$, by Lemma \ref{lem:continuity} we have
\[
0 \le \mathcal{E}\Big(\mathbf{Q}^n, \mathbf{Q}_{F_j,G_j,K_j,n_{k_j}} + \frac{r_{k_j}}{n_{k_j}} \mathbf{t} \Big) \ \ \longrightarrow \ \ \mathcal{E}\Big(\mathbf{Q}^n, \mathbf{Q}_{I,J,K,n} + \frac{r}{n} \mathbf{t} \Big) \ge 0.
\]
Therefore, by Proposition \ref{prop:hornexp}, $\mathbf{Q}^n$ are the spectral quantiles of Hermitian matrices $A+B=C$, while by Lemma \ref{lem:avg-L1}, $\mathbf{Q}^n \to \mathbf{Q}$, and thus $\mathbf{Q} \in \mathscr{H}(\mathbb{R})$.
\end{proof}

The inequalities \eqref{eqn:thin-horn} correspond to actual Horn inequalities for $n_k$-by-$n_k$ matrices, and the requirement that the values of $r_k / n_k$ be dense in $(0,1)$ merely ensures that we can approximate the \textit{shifted} triples $\mathbf{Q}_{I,J,K,n} + \frac{r}{n} \mathbf{t}$ corresponding to all of the other Horn inequalities.  On the other hand, if we allow ourselves to shift triples of quantile functions by arbitrary multiples of $\mathbf{t}$ as in Theorem \ref{thm:SC}, then we can recover the entire system only from triples for which $r_k/n_k \to 0$ arbitrarily quickly.

\begin{cor}
    Choose any sequences $(n_k)_{k \ge 1}$, $(r_k)_{k \ge 1}$ of positive integers such that $r_k < n_k$, $r_k \to \infty$, and $r_k / n_k \to 0$. Let $\mathbf{Q}$ be a triple of integrable quantile functions satisfying $\mathrm{tr}(\mathbf{Q}) = 0$. Then $\mathbf{Q} \in \mathscr{H}(\mathbb{R})$ if and only if
    \begin{equation*}
    \mathcal{E}\big(\mathbf{Q}, \mathbf{Q}_{I,J,K,n_k} + \mu \mathbf{t} \big) \ge 0 \quad \text{ for all $(I,J,K) \in T^{n_k}_{r_k}$, $k \ge 1$, and $0 < \mu \le \eta_{\mathbf{Q}_{I,J,K,n_k}}$.}
    \end{equation*}
\end{cor}

The proof is very similar to that of Theorem \ref{thm:redundancy}, and we omit it.

\section{Self-characterising sets and the question of uniqueness}
\label{sec:SCunique}

 In this section, we begin to develop a theory of sets that describe themselves in a fairly general sense.  Concretely, we consider subsets $A \subset X$ where $X$ is an ambient set with a relation, such that relatedness with every element of $A$ is a necessary and sufficient condition for an element of $X$ to belong to $A$. We introduce and study both a weak and a strong form of self-characterisation.

 The motivation for this theory is the question: how close does the self-characterisation property (Theorem \ref{thm:SC}) come to uniquely determining the asymptotic Horn system $\mathscr{H}[0,1]$?  It is too much to hope that Theorem \ref{thm:SC} alone could be used to \textit{define} $\mathscr{H}[0,1]$, but it can form part of a definition when combined with some additional information, and we prove a modest uniqueness result in this direction. We also prove that weak self-characterisation by itself does not allow any reduction of the relevant Horn inequalities, in the sense that if some subset of Horn inequalities implies all of the others under the assumption that the full set of inequalities is weakly self-characterising, then that subset implies the others even without the assumption of self-characterisation.  We conjecture that a similar statement is true for strong self-characterisation.  It remains possible that weak or strong self-characterisation plus further assumptions of a geometric or topological nature (such as the self-averaging property and the fact that $n$-integral and $r$-atomic points are dense in $\mathscr{H}[0,1]$) might be sufficient to uniquely determine the asymptotic Horn system.

\subsection{Generalities}
We turn now to the more abstract setting of a set endowed with a relation.

\begin{definition}
    Let $X$ be a set, and let $L \subset X \times X$ be a relation on $X$. We write $xLy$, or say $x$ \emph{likes} $y$, if $(x,y) \in L$. Otherwise we say $x$ \emph{dislikes} $y$. We say that $x$ is \emph{self-liking} if $xLx$, otherwise $x$ is \emph{self-disliking}.

We say that a subset $A \subset X$ is \emph{friendly} (under $L$) if
    \begin{equation}
        \forall x, y \in A \ \ xLy.
    \end{equation}
    
    We say that $A$ is \emph{weakly packed} (under $L$) if 
    \begin{equation}
        \forall x \in X \ \ \big( (\forall y \in A \cup \{x\} \ \ xLy \ \& \ yLx) \implies x \in A \big).
    \end{equation}
Equivalently, $A$ is weakly packed if every $x$ not in $A$ either dislikes or is disliked by some $y$ in $A$, or is self-disliking. 

We say that $A$ is \emph{strongly packed} (under $L$) if 
    \begin{equation}
        \forall x \in X \ \ \big( (\forall y \in A \ \ xLy) \implies x \in A \big).
    \end{equation}
Equivalently, $A$ is strongly packed if every $x$ not in $A$ dislikes some $y$ in $A$.  Clearly, a set that is strongly packed is also weakly packed, but if $L$ is not symmetric and reflexive then the converse is not necessarily true.

    We say that $A$ is \emph{weakly self-characterising} (under $L$) if $A$ is both
friendly and weakly packed under $L$. Equivalently, 
\begin{equation} \label{eqn:weak-sc-def}
A = \{ x \in X \ | \ \forall y \in A \cup \{x\} \ \ xLy \ \& \ yLx\}.
\end{equation}
Analogously, we say that $A$ is \emph{strongly self-characterising} (under $L$) if $A$ is both friendly and strongly packed under $L$, or equivalently,
\begin{equation} \label{eqn:strong-sc-def}
A = \{ x \in X \ | \ \forall y \in A \ \ xLy\}.
\end{equation}The name self-characterising refers to the fact that relatedness with all elements of a self-characterising set (or mutual relatedness together with self-liking) provides a necessary and sufficient condition for membership.
\begin{figure}
\begin{tikzpicture}[>=Stealth, node distance=2cm]
  \node (A) at (0,0) {$x_1$};
  \node (B) at (2,1) {$x_2$};
  \node (C) at (4,1) {$x_4$};
  \node (D) at (2,-1) {$x_3$};
  \node (E) at (3.5,-1) {$x_5$};

  \draw[->, OliveGreen, bend left] (A) to (B);
  \draw[->, OliveGreen, bend left] (A) to (B);
  \draw[->, OliveGreen, bend left] (A) to (B);

  \draw[->, OliveGreen, bend left] (A) to (B);
  \draw[->, OliveGreen, bend left] (B) to (A);
  \draw[->, OliveGreen, bend left] (B) to (D);
  \draw[->, OliveGreen, bend left] (D) to (B);
  \draw[->, OliveGreen, bend left] (A) to (D);
  \draw[->, OliveGreen, bend left] (D) to (A);

  \draw[->, OliveGreen, bend left] (C) to (B);
  \draw[->, OliveGreen, bend left] (C) to (E);
  \draw[->, OliveGreen, bend left] (E) to (D);

  \draw[->, OliveGreen, loop above, looseness=10] (A) to (A);
  \draw[->, OliveGreen, loop above, looseness=10] (B) to (B);
  \draw[->, OliveGreen, loop above, looseness=10] (D) to (D);
  \draw[->, OliveGreen, loop above, looseness=10] (C) to (C);

  \draw[dashed, gray] (1.3,0.2) circle (1.9cm);
  \node at (1.0,2.0) {$A$};
\end{tikzpicture}

\caption{A relation on the set $X = \{x_1,\ldots,x_5\}$. We draw a green arrow from $x_i$ to $x_j$ if $x_i$ likes $x_j$. The subset $A = \{x_1,x_2,x_3\}$ is strongly self-characterising, as all elements of $A$ like each other and themselves, and every element not in $A$ dislikes an element of $A$.}
\end{figure}

    If $A, B \subseteq X$ are two subsets, then we say that $A$ is weakly or strongly packed (resp. self-characterising) under $L$ \emph{relative to $B$} if $A \cap B$ is weakly or strongly packed (resp. self-characterising) under the restriction of $L$ to $B$.  There is no need for a notion of friendliness relative to $B$, since the property of friendliness depends only on $A$ and not on a choice of ambient set.

    We will sometimes suppress explicit mention of the relation $L$ when it is clear from context.
\end{definition}

We stress that self-characterisation can only be an informative property once we have fixed a particular relation.  Otherwise it tells us nothing whatsoever about the set $A$.  Indeed, any $A \subset X$ is strongly self-characterising under the relation $L = A \times A \subset X \times X$.  On the other hand, if the set $X$ and the relation $L$ have some additional structure, then self-characterisation may automatically guarantee other properties, as illustrated in the following easy example.

\begin{prop} \label{prop:sc-closed}
    Let $X$ be a topological space, let $f: X \times X \to \R$ be a continuous function, and consider the relation $L$ defined by $x L y \iff f(x,y) \ge 0$. If $A \subset X$ is weakly self-characterising, then $A$ is closed.
\end{prop}

\begin{proof}
Since $A$ is weakly self-characterising, we have
\begin{align*}
    A &= \{ x \in X \ | \ \forall y \in A \ \ f(x,y), f(y,x), f(x,x) \ge 0 \} \\ &= \bigcap_{y \in A} \{ x \in X \ | \ f(x,y) \ge 0\} \cap \bigcap_{y \in A} \{ x \in X \ | \ f(y,x) \ge 0\} \cap \{ x \in X \ | \ f(x,x) \ge 0\}.
\end{align*}
Since $f$ is continuous, the functions $x \mapsto f(x,y)$, $x \mapsto f(y,x)$, and $x \mapsto f(x,x)$ are continuous as well, and thus their upper contour sets are closed. Therefore $A$ is an intersection of closed sets and must itself be closed.
\end{proof}

Moreover, for certain types of relations, self-characterisation is equivalent to more familiar properties.

\begin{example}
    If $L$ is an equivalence relation, then weakly self-characterising subsets are also strongly self-characterising and are precisely the equivalence classes of $L$.
\end{example}

\begin{example} \label{ex:cones}
    If $X$ is a Euclidean space and $L$ is the relation defined by
    \[
    xLy \quad \iff \quad \langle x,y \rangle \ge 0,
    \]
    then the self-characterising subsets are the self-dual cones in $X$. Again there is no distinction between weak and strong self-characterisation in this example, since the inner product is symmetric and positive definite, and therefore $L$ is symmetric and reflexive.
\end{example}

Friendly, packed, and self-characterising sets satisfy some convenient properties with regard to containment of subsets.

\begin{prop} \label{prop:sc-props}
    Let $X$ be a set endowed with a relation $L$, and let $A \subsetneq B \subset X$ be subsets.

    \begin{enumerate}
        \item If $B$ is friendly, then $A$ is also friendly and is not weakly (or strongly) packed.
        \item If $A$ is weakly (resp. strongly) packed, then $B$ is also weakly (resp. strongly) packed and is not friendly. 
        \item The weakly self-characterising subsets of $X$ are precisely the sets that are friendly and are not contained in another friendly set (i.e.~the sets that are maximal under containment among friendly sets),  or equivalently the sets that are weakly packed and do not strictly contain another weakly packed set (i.e.~the sets that are minimal under containment among weakly packed sets).
    \end{enumerate}
\end{prop}

\begin{proof}
To prove (1), suppose that $B$ is friendly. For any $x,y \in A$, we also have $x,y \in B$, and therefore $xLy$, so that $A$ is friendly as well.  On the other hand, since $A \ne B$, there must be some element $x \in B - A$.  We thus have $xLy$ and $yLx$ for all $y \in A$, but $x \not \in A$, so $A$ is not weakly packed.

To prove (2), suppose that $A$ is weakly (resp. strongly) packed.  For any $x \in X$, if $xLy$ (resp. $xLy$, $yLx$ and $xLx$) for all $y \in B$, then in particular $xLy$ (resp. $xLy$ and $yLx$) for all $y \in A$. Since $A$ is weakly (resp. strongly) packed, we must then have $x \in A$, and thus also $x \in B$, and therefore $B$ is weakly (resp. strongly) packed as well.  On the other hand, for $x \in B - A$, either $x$ must be self-disliking or there must be some $y \in A$ that either dislikes or is disliked by $x$, since otherwise $A$ would not be weakly packed.  Therefore $B$ is not friendly.

To prove (3), suppose that $A$ is a weakly self-characterising set.  Then $A$ is both weakly packed and friendly, and thus property (1) guarantees that $A$ cannot strictly contain another weakly packed set, while (2) guarantees that $A$ cannot be strictly contained in another friendly set.  Next suppose that $A$ is a friendly set that is maximal under containment, and suppose for the sake of contradiction that $A$ is not weakly self-characterising (i.e. not weakly packed).  Then there must exist some $x \in X - A$ that likes itself and also likes and is liked by all elements of $A$.  But then $A \cup \{x\}$ would be a friendly set that strictly contains $A$, contradicting the assumption that $A$ is maximal.  Finally, suppose that $A$ is a weakly packed set that is minimal under containment, and suppose for the sake of contradiction that $A$ is not friendly. Then there must be some $x, y \in A$, possibly not distinct, such that $x$ dislikes $y$.  But then $A - \{x\}$ would be a weakly packed strict subset of $A$, contradicting the assumption that $A$ is minimal.
\end{proof}

We now turn to questions of existence and uniqueness of self-characterising sets.  In the proof of the following corollary, in the setting where $X$ is infinite we assume the axiom of choice.

\begin{cor} \label{cor:existence}
    Let $X$ be a set, and $L$ a relation on $X$. Then $X$ contains at least one weakly self-characterising set.

More generally, if $E$ is any friendly subset of $X$, there exists a weakly self-characterising set containing $E$. 
\end{cor}

\begin{proof}
First observe that the empty subset is friendly under any relation on any set, so the friendly subsets of $X$, with the partial order of containment, form a non-empty poset.  If $X$ is finite, then this poset must contain a maximal element, which by Proposition \ref{prop:sc-props} is a weakly self-characterising subset.  The same argument applies if we restrict our attention to the poset of friendly subsets containing a given friendly subset $E$.

If $X$ is infinite, then the same conclusion follows by invoking Zorn's lemma. If $A_1 \subseteq A_2 \subseteq \hdots \subseteq X$ is an ascending chain of friendly subsets, then the union $\bigcup_{n=1}^\infty A_n$ is also friendly and is an upper bound for the chain.  By Zorn's lemma, there is thus a friendly subset of $X$ that is maximal under containment, and by Proposition \ref{prop:sc-props} this subset is weakly self-characterising.
\end{proof}

We can also ask when there exists a \emph{unique} weakly self-characterising set containing a given friendly set.  The proof of the following proposition again requires the axiom of choice when the set $X$ is infinite.

\begin{prop} \label{prop:unique-weak-sc}
Let $L$ be a relation on a set $X$, let $E$ be a friendly subset of $X$, and let $V$ be a weakly self-characterising set containing $E$. The following are equivalent:
\begin{enumerate}
\item $V$ is the unique weakly self-characterising set containing $E$.
\item For $x \in X$ self-liking,
\begin{align} \label{eq:impl1}
\big( \forall y \in E \ \ xLy~\&~yLx \big) \iff \big( \forall y \in V \ \ xLy~\&~yLx \big).
\end{align}
\end{enumerate}
\end{prop}
\begin{proof}
First we prove that (2) implies (1) by considering the contrapositive: we show that if (1) does not hold, then (2) does not hold.
In this direction, suppose there are distinct weakly self-characterising sets $V$ and $W$, both of which contain $E$. By Proposition 3.2, neither $V$ nor $W$ contains the other. In particular, there exists $x \in W-V$. Since $W$ is friendly, $x$ is self-liking and both likes and is liked by all elements of $W$, and in particular, $x$ likes and is liked by all elements of $E$. However, $x$ cannot both like and be liked by all elements of $V$, since otherwise $x$ would have to belong to $V$ because $V$ is weakly packed. Thus \eqref{eq:impl1} does not hold.

We now show (1) implies (2). Suppose $V$ is the unique weakly self-characterising set containing $E$, and suppose $x$ is some self-liking element of $X$ satisfying $xLy$ and $yLx$ for all $y \in E$. If $x$ itself is in $V$, then $x$ likes and is liked by all elements of $V$, and we are done. 

We now show that if $x$ is not in $V$, then we contradict (1). Indeed, if $x$ is not in $V$, then $x$ is a self-liking element that likes and is liked by all elements of $E$. Thus $E \cup \{x\}$ is friendly, and by Corollary \ref{cor:existence}, there exists some weakly self-characterising set $W$ containing $E \cup \{x\}$. Note that $x \in W-V$, thus $W$ and $V$ are distinct weakly self-characterising sets containing $E$, a contradiction.
\end{proof}

Taking $E$ to be the empty set in Proposition \ref{prop:unique-weak-sc}, the following is immediate:

\begin{cor}
    Let $L$ be a relation on a set $X$, and suppose that $V \subseteq X$ is the unique weakly self-characterising subset under $L$. Then $V = \{ x \in X \ | \ xLx \}.$
\end{cor}

The existence and uniqueness of strongly self-characterising sets is a different matter altogether.  An analogous statement to Corollary \ref{cor:existence} does \emph{not} hold for strongly self-characterising sets: it is not true in general that every friendly set is contained in a strongly self-characterising set, and indeed strongly self-characterising sets may not exist at all, as shown in the example in Figure \ref{fig:no-strong-sc}.

\begin{figure}[H]
\begin{tikzpicture}[>=Stealth, node distance=2cm]
  \node (A) at (0,0) {$x$};
  \node (B) at (2,0) {$y$};

  \draw[->, OliveGreen, bend left] (B) to (A);

  \draw[->, OliveGreen, loop above, looseness=10] (A) to (A);
\end{tikzpicture}

\caption{A relation on the set $X = \{x,y\}$ for which there are no strongly self-characterising sets. (A green arrow from $a$ to $b$ denotes $aLb$.) The friendly sets are the empty set and the singleton $\{x\}$, while the strongly packed sets are the singleton $\{y\}$ and the full set $\{x,y\}$.} \label{fig:no-strong-sc}
\end{figure}

\subsection{The Horn inequalities and the question of uniqueness}

We now apply the preceding discussion to the Horn inequalities.

Let $\mathcal{M}$ be the space of triples $\mathbf{Q} = (Q_1, Q_2, Q_3)$ of quantile functions of probability measures supported on $[0,1]$ and satisfying $\tr(\mathbf{Q}) = 0$,  and consider the relation $\Vdash$ on $\mathcal{M}$ defined by
\[
\mathbf{Q} \Vdash \btQ \quad \iff \quad \mathcal{E}(\mathbf{Q}, \btQ + \mu \mathbf{t}) \ge 0 \quad \forall \, \mu \in [0, \eta_\btQ].
\]

Define
\begin{align} \label{eqn:Ucal-def}
\begin{split}
\mathcal{U}_n &= \bigcup_{m=2}^n \bigcup_{r=1}^{m-1} \Big\{ \, \mathbf{Q}_{I,J,K,m} \ \ \Big | \ \ (I,J,K) \in U^m_r \, \Big\}, \\
\mathcal{U} &= \bigcup_{n=2}^\infty \mathcal{U}_n,
\end{split}
\end{align}
with $U^n_r$ defined by (\ref{eqn:Unr-def}) and $\mathbf{Q}_{I,J,K,m}$ defined by (\ref{eq:Iquant}) and (\ref{eqn:QIJK-def}).  Then we have the following modest uniqueness result.

\begin{thm} \label{thm:H-sc-unique}
        The set $\mathscr{H}[0,1] \subset \mathcal{M}$ is strongly self-characterising under the relation $\Vdash$. Moreover, $\mathscr{H}[0,1]$ is the unique weakly self-characterising subset of $\mathcal{M}$ that is also weakly self-characterising relative to $\mathcal{U}_n$ for all $n \ge 2$.
\end{thm}

\begin{proof}
The fact that $\mathscr{H}[0,1]$ is strongly self-characterising under $\Vdash$ is a direct restatement of the criterion for membership in $\mathscr{H}[0,1]$ in Theorem \ref{thm:SC}.

Observe that
\[
\mathscr{H}[0,1] \cap \mathcal{U}_n = \Big\{ \mathbf{Q}_{I,J,K,m} \ \ \Big| \ \ (I,J,K) \in T^m_r \text{ for some $1 \le r < m \le n$} \Big\},
\]
and that for any such triple $(I,J,K)$, $\mathbf{Q}_{I,J,K,m} + \frac{r}{m} \mathbf{t} \in \mathscr{H}[0,1)$. Thus for any $\mathbf{Q} \in \mathcal{U}_n$, by Theorem \ref{thm:SC} and Propositions \ref{prop:hornexp} and \ref{prop:forward},
\begin{align*}
\mathbf{Q} \Vdash \btQ \quad \forall \, \btQ \in \mathscr{H}[0,1] \cap \mathcal{U}_n \quad \iff \quad \mathcal{E}\Big(\mathbf{Q}, \mathbf{Q}_{I,J,K,m} + \frac{r}{m} \mathbf{t} \Big) \ge 0 \quad \forall \, (I,&J,K) \in T^m_r, \\ & 1 \le r < m \le n.
\end{align*}
The fact that $\mathscr{H}[0,1]$ is weakly self-characterising relative to $\mathcal{U}_n$ for all $n \ge 2$ then follows from Lemma \ref{lem:Tmember} and Proposition \ref{prop:forward}.

To see that $\mathscr{H}[0,1]$ is the unique weakly self-characterising subset of $\mathcal{M}$ with this property, let $S \subset \mathcal{M}$ be any weakly self-characterising set that is also weakly self-characterising relative to $\mathcal{U}_n$ for all $n \ge 2$.  We will show that $S = \mathscr{H}[0,1]$.

First observe that it suffices to prove that $S \cap \mathcal{U}_n = \mathscr{H}[0,1] \cap \mathcal{U}_n$.   We then have
\[
\bigcup_{n=2}^\infty \mathscr{H}[0,1] \cap \mathcal{U}_n \subseteq S,
\]
and since $S$ is friendly, every $\mathbf{Q} \in S$ must satisfy
\[
\mathcal{E}\Big(\mathbf{Q}, \mathbf{Q}_{I,J,K,m} + \frac{r}{m} \mathbf{t} \Big) \ge 0 \qquad \forall \, (I,J,K) \in T^m_r, \ m \ge 2, \ 1 \le r < m. 
\]
By Proposition \ref{prop:forward} and \eqref{eq:neq}, we find that the same must hold for each averaged triple $\mathbf{Q}^n$, and thus $\mathbf{Q}^n$ represents the empirical spectral measures of an $n$-by-$n$ Hermitian triple $A_n + B_n = C_n$ with all eigenvalues in $[0,1]$.  Since $\mathbf{Q}^n \to \mathbf{Q}$, we find that $\mathbf{Q}$ is a limit of spectra of Hermitian triples, and thus $\mathbf{Q} \in \mathscr{H}[0,1]$, implying $S \subseteq \mathscr{H}[0,1]$. Since one weakly self-characterising set cannot strictly contain another by Proposition \ref{prop:sc-props}, we can then conclude that $S = \mathscr{H}[0,1]$.

Therefore it only remains to prove that $S \cap \mathcal{U}_n = \mathscr{H}[0,1] \cap \mathcal{U}_n$, which we do by induction. The base case is $n=2$. From  (\ref{eqn:Unr-def}) we have
\[
U^2_1 = \Big\{ \big(\{1\}, \{1\}, \{1\}\big), \ \big(\{1\}, \{2\}, \{2\}\big), \ \big(\{2\},\{1\},\{2\}\big)\Big\},
\]
corresponding to the Weyl inequalities for 2-by-2 matrices. A direct calculation shows that the set
\[
\mathcal{U}_2 = \big\{\mathbf{Q}_{I,J,K,2} \ \big| \ (I,J,K) \in U^2_1 \big\} = \big\{\mathbf{Q}_{I,J,K,2} \ \big| \ (I,J,K) \in T^2_1 \big\}
\]
is friendly under $\Vdash$, and since weakly self-characterising sets are maximal friendly sets we must then have $S \cap \mathcal{U}_2 = \mathscr{H}[0,1] \cap \mathcal{U}_2 = \mathcal{U}_2$.

Suppose now that $S \cap \mathcal{U}_k = \mathscr{H}[0,1] \cap \mathcal{U}_k$ for some $k \ge 2$. We will show that $S \cap \mathcal{U}_{k+1} = \mathscr{H}[0,1] \cap \mathcal{U}_{k+1}$. Again by by Proposition \ref{prop:sc-props}, it suffices to show that $S \cap \mathcal{U}_{k+1} \subseteq \mathscr{H}\cap \mathcal{U}_{k+1}$. Let $\mathbf{Q} \in S \cap \mathcal{U}_{k+1}$.  Then $\mathbf{Q} \Vdash \btQ$ for all $\btQ \in \mathscr{H}[0,1] \cap \mathcal{U}_{k}$, since $\mathscr{H}[0,1] \cap \mathcal{U}_{k} = S \cap \mathcal{U}_{k} \subset S\cap \mathcal{U}_{k+1}$, and $S \cap \mathcal{U}_{k+1}$ is friendly.  By Lemma \ref{lem:Tmember} and Proposition \ref{prop:forward}, this means that $\mathbf{Q} = \mathbf{Q}_{I,J,K,m}$ for some $(I,J,K) \in T^m_r$ with $1 \le r < m \le k+1$, so that $\mathbf{Q} \in \mathscr{H}\cap \mathcal{U}_{k+1}$, and thus $S\cap \mathcal{U}_{k+1} \subseteq \mathscr{H}\cap \mathcal{U}_{k+1}$ as desired. 
\end{proof}

The proof of Theorem \ref{thm:H-sc-unique} relies on the inductive structure of the sets $T^n_r$, but one could hope to find a condition that uniquely determines $\mathscr{H}[0,1]$ as a self-characterising set without exploiting this structure.  For instance, one could look for subsets $E \subset \mathscr{H}[0,1]$ such that $\mathscr{H}[0,1]$ is the unique set that is (weakly or strongly) self-characterising and contains $E$.  Such subsets clearly exist, since the property holds if we take $E = \mathscr{H}[0,1] \cap \mathcal{U}$ to be the set of all triples representing Horn inequalities. But this trivial choice cannot be minimal, since the Horn inequalities themselves are known to be redundant for $n \ge 5$ (see \cite{BuchSC, Fu}), and the property must also hold if we take $E$ to be a maximal nonredundant subset of $\mathscr{H}[0,1] \cap \mathcal{U}$. In fact, assuming the axiom of choice, Proposition \ref{prop:unique-weak-sc} above immediately implies that if we are interested in describing $\mathscr{H}[0,1]$ as the unique \emph{weakly} self-characterising set containing some given set $E$, then finding such a set $E$ that is minimal under containment is equivalent to finding a maximal nonredundant subset of Horn inequalities:

\begin{cor}
Let $E \subset \mathscr{H}[0,1]$ be a subset that uniquely determines $\mathscr{H}[0,1]$ as a weakly self-characterising set, in the sense that if $S \subset \mathcal{M}$ is weakly self-characterising and $E \subset S$, then $S = \mathscr{H}[0,1]$. Then for all $\mathbf{Q} \in \mathcal{M}$,
\[
\mathbf{Q} \Vdash \btQ \quad \forall \, \btQ \in E \quad \iff \quad \mathbf{Q} \Vdash \btQ \quad \forall \, \btQ \in \mathscr{H}[0,1].
\]
\end{cor}

We conjecture that a similar statement holds if one considers strong rather than weak self-characterisation.

\section{Directions for further work}
\label{sec:further}

Even with the above results on self-characterisation, dense subsets, and other properties of the set $\mathscr{H}[0,1]$, the geometry of the asymptotic Horn system remains fairly mysterious, and we suspect the reader may agree that the preceding analysis has raised as many questions as it has answered.  To close, we point out some directions for future research that would be natural sequels to the investigation in this paper.

First, some of the results in this article can likely be improved. For example, it is not immediately apparent whether Theorem \ref{thm:redundancy} on the redundancy in the Horn inequalities in the $n \to \infty$ limit is sharp. It is natural to ask whether a converse statement holds: given sequences $r_k < n_k$ such that $(r_k / n_k)_{k \ge 1}$ is \textit{not} dense in $(0,1)$, can one construct Horn inequalities in some $T^n_r$ that are \textit{not} implied by the inequalities in $(T^{n_k}_{r_k})_{k \ge 1}$?  And we have not touched the more detailed question of which further inequalities in each $T^{n_k}_{r_k}$ could additionally be discarded without changing the set of solutions.

One could also hope for a stronger uniqueness result for the asymptotic Horn system, showing that it is the unique subset of $\mathcal{M}$ that is weakly or strongly self-characterising under $\Vdash$ (or some other relation expressing similar inequalities) and that also satisfies some additional properties. Examples of such properties could include the self-averaging property stated in Proposition \ref{prop:selfavg}, or a version of the lattice-point approximation property stated in Theorem \ref{thm:new23}.

Another possible direction of inquiry is to relate the present work to probabilistic versions of Horn's problem in the $n \to \infty$ limit.  Recent work in random matrix theory has studied random matrix ensembles that can be regarded as quantitative versions of Horn's problem, in which one would like to know not only whether a given triple $(\alpha,\beta,\gamma) \in \R^{3n}$ can occur as the spectra of Hermitian matrices $A+B=C$, but also, in a certain sense, how many such matrices there are \cite{CMZ,CMZ2,CM-Projections}.  It is natural to wonder whether one could formulate a large-$n$ version of the randomised Horn's problem in which $\mathscr{H}[0,1]$ or the asymptotic Horn bodies $\mathscr{H}_{\mu,\nu}$ defined in (\ref{eqn:asymp-body-def}) would play the role of the finite-dimensional Horn polytopes.  This question is also related to recent work of Narayanan, Sheffield and Tao on asymptotics of random hives \cite{NaShe, NaSheTa}, since the probability measure studied in the randomised Horn's problem is the pushforward by a linear projection of the uniform probability measure on a hive polytope.

Finally, Horn's problem and its probabilistic version can both be viewed from a much more general perspective; namely, they are special cases of the problems of determining the moment polytope and Duistermaat--Heckman measure of a symplectic manifold equipped with a Hamiltonian group action \cite{DH, Knut}.  A number of other well-known random matrix ensembles can be viewed in the same light, including the randomised Schur's problem, the randomised quantum marginals problem, uniform random antisymmetric matrices with deterministic eigenvalues, and the Hermitian orbital corners process \cite{CM-Projections, CMZ2, CuencaOrbital, MaMc, CDKW}.  Asymptotic questions about all of the above models can be interpreted as problems in high-dimensional symplectic geometry, and one could try to ask and answer interesting questions about moment polytopes and Duistermaat--Heckman measures arising from actions of high-dimensional Lie groups in a more general setting than the particular case of Horn's problem.

\section*{Acknowledgements}
The work of C.M. is partially supported by the National Science Foundation under grant number DMS-2103170, by the National Science and Technology Council of Taiwan under grant number 113WIA0110762, and by a Simons Investigator award via Sylvia Serfaty.  C.M. would like to thank Beno\^it Collins, Robert Coquereaux, and Jean-Bernard Zuber for helpful discussions of an early version of this article.  Part of this work was conducted at the Random Theory 2024 workshop in Estes Park, CO.

\bibliography{refs}
\bibliographystyle{amsplain}

\end{document}